\documentclass[12pt]{article}

 \usepackage[latin1]{inputenc}
\usepackage{fullpage}
\usepackage{amsfonts}
\usepackage{amsmath}
\usepackage{amssymb}
\usepackage{amsthm}
\usepackage{authblk}
\usepackage{mathrsfs}
\usepackage{mathrsfs}
\usepackage{tikz}
\usepackage[margin=1.1in]{geometry}
\usetikzlibrary{arrows,matrix}
\usepackage[backend = biber]{biblatex}
\usepackage{authblk}
\newcommand{\NN}{\mathbb{N}}

\newcommand{\RR}{\mathbb{R}}

\numberwithin{equation}{section}

\theoremstyle{plain}
\newtheorem{theorem}[equation]{Theorem}
\newtheorem{lemma}[equation]{Lemma}
\newtheorem{proposition}[equation]{Proposition}
\newtheorem{corollary}[equation]{Corollary}

\theoremstyle{definition}
\newtheorem{definition}[equation]{Definition}
\newtheorem{example}[equation]{Example}

\theoremstyle{remark}

\newtheorem*{example*}{Example}

 { \begin{list}%
         {$\bullet$}%
         {\setlength{\labelwidth}{20pt}%
          \setlength{\leftmargin}{25pt}%
          \setlength{\topsep}{0pt}
          \setlength{\itemsep}{1ex}
          \setlength{\parsep}{0pt}}}%
 { \end{list} }

\title{Tropical Coordinates on the Space of Persistence Barcodes}
\author{Sara Kali\v{s}nik}
\affil{Max Planck Institute for Mathematics in the Sciences, sara.kalisnik@mis.mpg.de}
\affil{Wesleyan University, skalisnikver@wesleyan.edu}
  \date{}
  
\begin{document}
 \maketitle

\begin{abstract}
The aim of applied topology is to use and develop topological methods for applied mathematics, science and engineering. One of the main tools is persistent homology, an adaptation of classical homology, which assigns a barcode, i.e.\ a collection of intervals, to a finite metric space. Because of the nature of the invariant, barcodes are not well-adapted for use by practitioners in machine learning tasks.  We can circumvent this problem by assigning numerical quantities to barcodes and these outputs can then be used as input to standard algorithms. It is the purpose of this paper to identify tropical coordinates on the space of barcodes and prove that they are stable with respect to the bottleneck distance and Wasserstein distances.
\end{abstract}
\section{Introduction}
In the past two decades, with the emergence of `big data', topology started playing a more prominent role in data analysis~\cite{topodata, pattern}. Topological ideas have inspired methods for visualizing complex datasets~\cite{mapper} as well as `measuring' the shape of data. Using the most famous example of the latter, persistent homology~\cite{ZC, elz-tps-02}, researchers have solved problems in sensor networks~\cite{VinEvader, adams}, medicine~\cite{aaronpaperonimaging, Ferri}, neuroscience~\cite{Chung2009, Giusti_Pastalkova_Curto_Itskov_2015, Curto2013} and gained insights into texture images~\cite{Klein}.

The output of persistent homology is a barcode, i.e.\ a collection of intervals. The unusual structure of the invariant makes the method hard to combine with  standard algorithms within machine learning. For this reason various attempts have been made to assign numerical quantities to barcodes or to send these objects into a Hilbert space through a feature map, where computations are easier~\cite{Reininghaus_2015_CVPR, Carriere:2015:STS:2853908.2853910, FabioComplex, landi, Bubenik:2015:STD:2789272.2789275, algfn, DBLP:journals/corr/ChepushtanovaEH15}. For example, Di Fabio and Ferri~\cite{FabioComplex, landi} assign complex vectors to barcodes, Bubenik~\cite{Bubenik:2015:STD:2789272.2789275} persistence landscapes. 

Adcock et al.~\cite{algfn} identified an algebra of polynomials on the barcode space that can be used as coordinates. The problem with these functions is that they are not stable (i.e.\ Lipschitz) with respect to the Bottleneck and Wasserstein $p$-distances usually used. This prompted us to search for other types of coordinates.  All the aforementioned distances on the barcode space are defined by matching intervals from one barcode to another and computing penalties that involve taking maxima. For this reason the max and min functions, i.e.\ tropical functions, seemed like a natural choice and as it turns out they are indeed more suitable given the underlying structure of the barcode space. 

We represent a barcode with exactly $n$ intervals as a vector $(x_1, d_1, x_2, d_2,, \ldots, x_n,d_n)$, where $x_i$ denotes the left endpoint of the $i$-th interval and $d_i$ its length. We assume that $x_i \geq 0$ for all $i$. This is not unreasonable since when constructing simplicial complexes from point clouds, the parameter is radius, which is nonnegative. This condition is also crucial later on in the construction when taking filtered inverse limits as it ensures a good behavior of certain maps when appropriately restricted. Since the ordering of the intervals does not matter, we take the orbit space, $B_n$, of the action of the symmetric group on $n$ letters on the product $([0, \infty) \times [0, \infty))^n$ given by permuting the coordinates. 
The barcode space, $B$, is the quotient
\[
\coprod_n B_n /_{\sim},
\]
where $\sim$ is generated by equivalences of the form 
\[
\{(x_1, d_1), (x_2, d_2), \ldots, (x_n, d_n) \} \sim \{(x_1, d_1), (x_2, d_2), \ldots, (x_{n-1}, d_{n-1})\},
\]
whenever $d_n=0$.

After a short review of tropical algebra and persistent homology in Sections 2 and 3, Section~\ref{sec:maxnplus} is devoted to establishing the properties of 2-symmetric max-plus polynomials that respect this equivalence relation. In particular, Theorem~\ref{maxplusgen} provides a list of generators for this semiring. Unfortunately, this condition is so limiting that the only functions satisfying it involve only lengths of intervals. While we prove that these are individually stable with respect to the bottleneck and Wasserstein distances, there are not enough of them to separate the barcodes. In fact, in contrast to ordinary polynomials, no finite set exists that separates barcodes in $B_n$ (Theorem~\ref{finsub}). This forces us to expand the semiring of observed functions to tropical rational functions. We find a countable generating set (Theorem~\ref{tropicalgen}) that separates the barcodes and prove that each function in this set is stable with respect to the bottleneck and Wasserstein distances (Theorems~\ref{tropratstab} and~\ref{tropratstab1}). These functions and their sums, minima and maxima can be used by researchers interested in analyzing datasets of shapes. In Section 8 we give an example that demonstrates how they can be used to classify digits from the MNIST dataset.

Of course, a natural question that arises is how to select finitely many functions from this infinite family that we identify. In the case when we deal with barcodes whose birth and death times only take finitely many values it is not hard to find finitely many functions that separate them (this is further discussed in Section 8 on a particular example). Even here we might run into trouble because the vectors we obtain might be very high dimensional. We are currently working on automating this step and using machine learning methods (for example, the Lasso method) on this collection of coordinate functions to select their weights.

\section{Tropical Functions}\label{sec:back}
This section reviews the material that first appeared in Symmetric and $r$-Symmetric Tropical Polynomials and Rational Functions~\cite{KalisnikSym}.

\subsection{Min-plus and Max-plus Polynomials}\label{back}
Tropical algebra is based on the study of the tropical semiring $(\RR \cup \{\infty \}, \oplus, \odot)$. In this semiring, addition and multiplication are defined as follows:
\[
\begin{array}{ccc}
a\oplus b := \min{(a, b)} &\, \textrm{and} \,& a\odot b := a+b.
\end{array}
\]
Both are commutative and associative. The times operator $\odot$ takes precedence when plus $\oplus$ and times $\odot$ occur in the same expression.
The distributive law holds:
\[
a \odot (b\oplus c) = a\odot b \oplus a\odot c.
\]
Moreover, the Frobenius identity (Freshman's Dream) holds for all powers $n$ in tropical arithmetic:
\begin{equation}\label{freshman}
(a \oplus b)^n = a^n \oplus b^n.
\end{equation}

Both arithmetic operations have a neutral element. Infinity is the neutral element for addition and zero is the neutral element for multiplication:
\[
x \oplus \infty = x \quad \textrm{ and }\quad x \odot 0 = x.
\]

Related to the tropical semiring is the arctic semiring $(\RR \cup \{-\infty \}, \boxplus, \odot)$, where multiplication of two elements is defined as before, but adding means taking their maximum instead of the minimum:
\[
\begin{array}{ccc}
a\boxplus b := \max{(a, b)} &\quad \textrm{and} \quad& a\odot b := a+b.
\end{array}
\]
Its operations are associative, commutative and distributive as in the tropical semiring. 

Let $x_1, x_2, \ldots, x_n$ be variables representing elements in the max-plus semiring.  A \emph{max-plus monomial expression} is any product of these variables, where repetition is allowed. By commutativity, we can sort the product and write monomial expressions with the variables raised to exponents.

A \emph{max-plus polynomial expression} is a finite linear combination of max-plus monomial expressions:
\[
p(x_1, x_2, \ldots, x_n) = a_1\odot x_1^{a_1^1} x_2^{a_2^1} \ldots x_n^{a_n^1} \boxplus a_2\odot x_1^{a_1^2} x_2^{a_2^2} \ldots x_n^{a_n^2}\boxplus \ldots \boxplus a_m\odot x_1^{a_1^m} x_2^{a_2^m} \ldots x_n^{a_n^m},
\]
Here the coefficients $a_1, a_2, \ldots a_m$ are real numbers and the exponents $a_j^i$ for ${1\leq j \leq n}$ and ${1\leq i \leq m}$ are nonnegative integers. 

The \emph{total degree of a max-plus expression} $p(x_1, x_2, \ldots, x_n)$ is
\[
\operatorname{deg} p = {\max_{1\leq i \leq m}(a_1^i+ a_2^i + \ldots + a_n^i)}.
\]

The passage from max-plus polynomial expressions to functions is not one-to-one. For example, 
\[
x_1^2 \boxplus x_2^2 = x_1^2 \boxplus x_2^2 \boxplus x_1x_2
\]
for all $x_1, x_2$ and therefore the functions defined by $x_1^2 \boxplus x_2^2 $ and $ x_1^2 \boxplus x_2^2 \boxplus x_1x_2$ are the same, though the expressions are formally different.

Considered as a function, $p\colon \RR^n \to \RR$ has the following three properties:
\begin{itemize}
\item
 $p$ is continuous,
\item
  $p$ is piecewise-linear, where the number of pieces is finite, and
  \item
 $p$ is convex.
  \end{itemize}
 Max-plus monomials are the linear functions with nonnegative integer coefficients.

Let  $p$ and $q$ be max-plus polynomial expressions. If
\[
p(x_1, x_2, \ldots, x_n) = q(x_1, x_2, \ldots, x_n)
\]
for all $(x_1, x_2, \ldots, x_n)\in (\RR \cup \infty)^n$, then $p$ and $q$ are \emph{functionally equivalent}. We write $p \sim q$. 

\emph{The minimal representation} of a max-plus polynomial $p$ is such a max-plus expression
\[
a_1\odot x_1^{i_1^1} x_2^{i_2^1} \ldots x_n^{i_n^1} \boxplus a_2\odot x_1^{i_1^2} x_2^{i_2^2} \ldots x_n^{i_n^2}\boxplus \ldots \boxplus a_m\odot x_1^{i_1^m} x_2^{i_2^m} \ldots x_n^{i_n^m}
\]
 functionally equivalent to $p$ that for each $1\leq j \leq m$ there exists a point ${(x_1, x_2, \ldots, x_n) \in \RR^n}$, so that 
\[
a_j+i_1^jx_1 +\ldots +i_n^jx_n > \max_{1\leq s\leq m, s \neq j} (a_s+i_1^s x_1 +\ldots + i_n^s x_n).
\]

\begin{definition}
\emph{Max-plus polynomials} are the semiring of equivalence classes of max-plus polynomial expressions with respect to $\sim$. In the case of $n$ variables we denote the semiring by $\mathtt{MaxPlus}[x_1, x_2, \ldots, x_n]$.
\end{definition}
We define min-plus polynomials expression as max-plus with $\boxplus$ replaced by $\oplus$. We define the degree of a min-plus polynomial expression analogously.
\begin{definition}
\emph{Min-plus polynomials} are the semiring of equivalence classes of min-plus polynomial expressions with respect to functional equivalence relation $\sim$. In the case of $n$ variables we denote the semiring by $\mathtt{MinPlus}[x_1, x_2, \ldots, x_n]$. 
\end{definition}
It can be shown that degrees of all max-plus (min-plus) expressions in the same equivalence class are the same and that it is therefore possible to define the degree of a max-plus (min-plus) polynomial.

\subsection{Rational Tropical Functions}\label{ratfun}

A \emph{tropical rational expression} $r$ is a quotient
\[
r(x_1, \ldots, x_n) = p(x_1, \ldots, x_n) \odot q(x_1, \ldots, x_n)^{-1} = p(x_1, \ldots, x_n)-q(x_1, \ldots, x_n),
\]
where $p$ and $q$ are min-plus polynomial expressions. 

\begin{definition}
The semiring of equivalence classes of tropical rational expressions with respect to the functional equivalence relation is $\mathtt{RTrop}[x_1, x_2, \ldots, x_n]$ and is called the \emph{semiring of rational tropical functions}. 
\end{definition}

We will need the following statement later on.
\begin{lemma}\label{decomposition}
A tropical rational function $r$ in $n$ variables gives a decomposition of $\RR^n$ into
   a family of closures of open sets on which the function is affine,
   and the boundaries of these domains are piecewise linear. 
\end{lemma}
\begin{proof}
A min-plus polynomial $p$ is a piecewise-linear concave function, and its domains of linearity
consist of the cells in a polyhedral subdivision $\Sigma_p$ as in \cite[Definition 2.5.5]{tropintro}. A tropical rational function $r$ has the form $p-q$ where $p$ and $q$ are both
min-plus polynomials. Let $\Sigma = \Sigma_p \wedge \Sigma_q$ be the common refinement of the 
corresponding polyhedral decompositions of $\RR^n$, as defined prior to \cite[equation (2.3.1)]{tropintro}.
Then $p-q$ is linear on each cell of $\Sigma$, and the boundaries of these cells are polyhedral balls or spheres.
\end{proof}

Since
\[
-\operatorname{min}(a, b) = \operatorname{max}(-a, -b),
\]
tropical rational expressions are composed by taking finitely many maxima and minima of linear functions, i.e.\ the set of tropical rational expressions is the smallest subset of functions $\RR^n\to \RR$ containing all constant maps and projections that is closed under taking finitely many $+$, $\min$ and $\max$. 

Conversely, any function from the latter set can be represented by an expression of the form $p \odot q^{-1}$, where $p$ and $q$ are tropical polynomial expressions. The algorithm to produce $p$ and $q$ is the usual one of adding fractions by finding a common denominator, but performed in tropical arithmetic~\cite{KalisnikSym}.

\begin{example}\label{ex}
Let $r(x_1, x_2) = x_2x_1^{-1} \oplus (x_2)^{-1} \oplus (x_2 x_1 \oplus x_1)^{-1}$. We can write
\[
\begin{array}{lclc}
r(x_1, x_2)&= &  x_2 x_1^{-1} \oplus (x_2)^{-1} \oplus (x_2 x_1 \oplus x_1)^{-1}&\\
&= &  x_2^2 (x_2 \oplus 0) (x_1x_2 (x_2 \oplus 0))^{-1} \oplus (x_1 (x_2 \oplus 0))(x_1x_2 (x_2 \oplus 0))^{-1} \oplus (x_2(x_2 x_1 \oplus x_1))^{-1}&\\
&= & (x_2^3 \oplus x_2^2 \oplus x_1x_2 \oplus x_1 \oplus  x_2) (x_1x_2^2\oplus x_1x_2)^{-1} & \\
&=&(x_2^3 \oplus x_1 \oplus x_2) \odot (x_1x_2^2 \oplus x_1x_2)^{-1}.
\end{array}
\]
\end{example}

As a consequence, any tropical rational expression (or equivalently, any map composed by taking finitey many maxima and minima of linear functions) in $x_1, \ldots, x_n$ can be written in ordinary arithmetic as
\[
\max_{i=1, \ldots, l_1} (\sum_{k=1}^n a_{k, i} x_k + c_i) - \max_{j=1, \ldots, l_2} (\sum_{k=1}^n s_{k, j} x_k + u_j)
\] 
for some $a_{k, i}$, $c_i$, $s_{k, j}$ and $u_j$ where $k\in \{1, \ldots, n\}, i\in \{1, \ldots, l_1\}$ and $j \in \{1, \ldots, l_2\}$. 

The functions contained in either $\mathtt{MinPlus}[x_1, x_2, \ldots, x_n]$, $\mathtt{MaxPlus}[x_1, x_2, \ldots, x_n]$ or $\mathtt{RTrop}[x_1, x_2, \ldots, x_n]$ are called \emph{tropical functions}.

\subsection{Symmetric and $2$-Symmetric Tropical Functions}
\begin{definition}
A tropical function $f$ is \emph{symmetric} if 
\[
f(x_1,\ldots, x_n) = f(x_{\pi(1)},\ldots ,x_{\pi(n)})
\]
for every permutation $\pi \in S_n$.
\end{definition}

Given variables $x_1, \ldots, x_n$, we define the \emph{elementary symmetric max-plus polynomials} $\sigma_1,\ldots, \sigma_n \in \mathtt{MaxPlus}[x_1, x_2, \ldots, x_n]$ by the formulas
\[
\begin{array}{lcl}
\sigma_1 &=& x_1 \boxplus \ldots \boxplus x_n, \\
&\vdots& \\
\sigma_k &=& \boxplus_{\pi\in S_n} x_{\pi(1)}  \odot \ldots \odot x_{\pi(k)},\\
&\vdots& \\
\sigma_n &=& x_1\odot x_2 \odot \ldots \odot x_n.
\end{array}
\]

The following version of the Fundamental Theorem of Symmetric Polynomials holds.
\begin{theorem}[Fundamental Theorem of Symmetric Max-Plus Polynomials]~\cite{KalisnikSym} \label{fundamentalminplus}
Every symmetric max-plus polynomial in $\mathtt{MaxPlus}[x_1, x_2, \ldots, x_n]$ can be written as a max-plus polynomial in the elementary symmetric max-plus polynomials $\sigma_1, \ldots, \sigma_n$.
\end{theorem}

A tropical function in $n$ variables is symmetric if it is invariant under the action of $S_n$ that permutes the variables. We can generalize this definition as follows: a tropical function in $nr$ variables, divided into n blocks of $r$ variables each, is $r$-symmetric if it is invariant under the action of $S_n$ that permutes the blocks while preserving the order of the variables within each block. We state the results for $r=2$ because persistence barcodes are collections of intervals.

\begin{definition}
A tropical function $p$ is \emph{2-symmetric} if 
\[
p(x_{1, 1}, x_{1, 2}, \ldots, x_{n, 1}, x_{n, 2}) = p(x_{\pi(1), 1}, x_{\pi(1), 2}, \ldots , x_{\pi(n), 1}, x_{\pi(n), 2})
\]
for every permutation $\pi \in S_n$.
\end{definition}

Fix $n$. Let the symmetric group $S_n$ act on the matrix of indeterminates
\[
X = \begin{pmatrix} 
x_{1, 1} & x_{1, 2} \\
x_{2, 1} & x_{2, 2} \\
\vdots & \vdots \\
x_{n, 1} & x_{n, 2} \\
\end{pmatrix}
\]
by left multiplication. 
 Let 
\[
\mathscr{E}_n = \left\{ \begin{pmatrix} 
e_{1, 1} & e_{1, 2} \\
e_{2, 1} & e_{2, 2} \\
\vdots & \vdots \\
e_{n, 1} & e_{n, 2} \\
\end{pmatrix} \neq [0]_n^2 \,\mid \, e_{i, j} \in \{0, 1\} \textrm{ for }i=1,2,\ldots, n, 
\textrm{ and } j=1,2  \right\}.
\]
A matrix $\begin{pmatrix} 
e_{1, 1} & e_{1, 2} \\
e_{2, 1} & e_{2, 2} \\
\vdots & \vdots \\
e_{n, 1} & e_{n, 2} \\
\end{pmatrix}\in \mathscr{E}_n$ determines a max-plus monomial $P(E) = x_{1, 1}^{e_{1,1}} x_{1, 2}^{e_{1,2}} \ldots x_{n, 1}^{e_{n,1}} x_{n, 2}^{e_{n, 2}}$.
We denote  the set of orbits under the row permutation action on $\mathscr{E}_n$ by $\mathscr{E}_n/{S_n}$. Each orbit $\{E_1, E_2, \ldots E_m\}$ determines a 2-symmetric max-plus polynomial 
\[P(E_1)\boxplus P(E_2)\boxplus \ldots \boxplus P(E_m).
\]
 Including the $[0]_n^2$ matrix in the definition of monomials would have been redundant as the 0 function can be expressed in terms of other 2-symmetric max-plus polynomials (by simply raising them to 0).

\begin{definition}\label{2symel}
We call the 2-symmetric max-plus polynomials that arise from orbits $\mathscr{E}_n/{S_n}$ \emph{elementary}. We let $\sigma_{(e_{1, 1}, e_{1, 2}), \ldots, (e_{n, 1}, e_{n, 2}) }$ denote the tropical polynomial that arises from the orbit
\[
\left[\begin{pmatrix} 
e_{1, 1} & e_{1, 2} \\
e_{2, 1} & e_{2, 2} \\
\vdots & \vdots \\
e_{n, 1} & e_{n, 2} \\
\end{pmatrix}\right].
\]
\end{definition}
\begin{example}
Let $n=2$. The 2-symmetric max-plus polynomials include
\[
\begin{array}{l}
\sigma_{(1, 0), (0, 0)} (x_1, d_1, x_2, d_2) = x_1 \boxplus x_2 = \max\{x_1, x_2 \}\\
\sigma_{(0, 1), (0, 0)} (x_1, d_1, x_2, d_2) = d_1 \boxplus d_2 = \max\{d_1, d_2 \}\\
\sigma_{(0, 1), (0, 1)} (x_1, d_1, x_2, d_2) = d_1\odot d_2  = d_1+d_2 \\
\sigma_{(1, 0), (1, 0)} (x_1, d_1, x_2, d_2) = x_1\odot x_2  = x_1+x_2 \\
\sigma_{(1, 1), (1, 1)} (x_1, d_1, x_2, d_2) = x_1 \odot d_1  \odot x_2  \odot d_2 = x_1+d_1+x_2 +d_2. \\
\end{array}
\]
More generally, for $n$ and $k\leq n$, $\sigma_{(0, 1)^k}$ is the total length of the $k$ longest bars and $\sigma_{(1, 0)^k}$ is the sum of the $k$ latest birth times.
\end{example}

There are enough 2-symmetric max-plus polynomials to separate the orbits. We will need this piece of information to show that the functions we define in Section~\ref{sec:troprat} separate barcodes.
\begin{proposition}~\cite{KalisnikSym}\label{sepbar}
Let  $\{ (x_{1}, y_{1}), \ldots, (x_{n}, y_{n})\}$ and $\{(x_{1}', y_{1}'), \ldots, (x_{n}', y_{n}')\}$ be two orbits under the row permutation action of $S_n$ on $\RR^{2n}$. If 
\[
\sigma(\{(x_{1}, y_{1}), \ldots, (x_{n}, y_{n})\}) = \sigma(\{x_{1}', y_{1}'), \ldots, (x_{n}', y_{n}')\})
\]
for all elementary 2-symmetric max-plus polynomials $\sigma$, then 
\[
\{(x_{1}, y_{1}), \ldots, (x_{n}, y_{n})\} = \{(x_{1}', y_{1}'), \ldots, (x_{n}', y_{n}')\}.
\]
\end{proposition}
We identified a finite set that generates symmetric max-plus polynomials. No such statement holds in the case of 2-symmetric max-plus polynomials~\cite{KalisnikSym}.
\begin{proposition}
The semiring of 2-symmetric max-plus polynomials in variables $x_{1, 1}, x_{1, 2}$, \newline $x_{2, 1}, x_{2, 2}$ is not finitely generated.
\end{proposition}

\section{Persistent Homology}\label{sec:ph}
Classical topologists developed homology in order to `measure' shape. In simplest terms, homology counts the occurrences of patterns, such as the number of connected components, loops and voids. The adaptation of homology to the study of point cloud data and more generally, filtrations of simplicial complexes, is persistent homology~\cite{Frosini_sizetheory, elz-tps-02, ZC}.

The motivating idea is that the union of discs with radius $r$ centered around points from the data set approximates the underlying shape of the point cloud. We do not know \textit{a priori} how to choose the radius. Persistent homology computes and keeps track of the changes in the homology these unions of discs over a range of radii parameters $r$. The output is a barcode, ie.\ a collection of intervals. Each interval corresponds to a topological feature which appears at the value of a parameter given by the left endpoint of the interval and disappears at the value given by the right endpoint. These barcodes play an analogous role as a histogram would in summarizing the shape of the data --- long intervals correspond to strong topological signals and short ones may correspond to noise.

\subsection{Barcode Space, Bottleneck Distance, Wasserstein Distances}
Each barcode with $n$ intervals can be encoded as $(x_1, d_1, x_2, d_2,\ldots, x_n, d_n)$ where $x_i$ is the left endpoint of the $i$-th interval and $d_i$ its length. Since the ordering of the intervals does not matter, we consider the orbit space of the action of the symmetric group on $n$ letters on the product $([0, \infty) \times [0, \infty))^n$ given by permuting the coordinates. We denote it by $B_n$. 

The barcode space $B$ is the quotient
\[
\coprod_n B_n /_{\sim},
\]
where $\sim$ is generated by equivalences of the form 
\[
\{(x_1, d_1), (x_2, d_2), \ldots, (x_n, d_n) \} \sim \{(x_1, d_1), (x_2, d_2), \ldots, (x_{n-1}, d_{n-1})\},
\]
whenever $d_n=0$.

Before specifying the distance between two barcodes, we specify the distance between any pair of intervals, as well as the distance between any interval and the set of zero length intervals $\Delta = \{(x, x)\,|\, 0 \leq x < \infty \}$.  Set
\[
\textrm{d}_\infty ((x_1, d_1), (x_2, d_2)) = \max (|x_1-x_2|, |d_1-d_2 +x_1- x_2|).
\]
The distance between an interval and the set $\Delta$ is
\[
\textrm{d}_\infty ((x, d), \Delta) = \frac{d}{2}.
\]
Let $\mathscr{B}_1 = \{I_\alpha\}_{\alpha \in A}$ and $\mathscr{B}_2 = \{J_\beta\}_{\beta \in B}$ be barcodes. For finite sets $A$ and $B$, and any bijection $\theta$ from a subset $A' \subseteq A$ to $B' \subseteq B$, the penalty of $\theta$, $P_\infty(\theta)$, is
\[
P_\infty (\theta) = \max(\max_{a\in A'}(\textrm{d}_\infty(I_a, J_{\theta(a)})), \max_{a\in A \setminus A'} \textrm{d}_\infty (I_a, \Delta), \max_{b\in B\setminus B'} \textrm{d}_\infty (I_b, \Delta)).
\] 
The \emph{bottleneck distance}~\cite{Cohen-Steiner_2007} is
\[
\textrm{d}_\infty(\mathscr{B}_1, \mathscr{B}_2) = \min_\theta P_\infty(\theta),
\]
where the minimum is over all possible bijections from subsets of $A$ to subsets of $B$. 

There are other metrics also commonly used  for barcode spaces. Setting the penalty for $\theta$ for $p\geq 1$ to
\[
P_p(\theta) = \sum_{a\in A'}\textrm{d}_\infty(I_a, J_{\theta(a)})^p +\sum_{a\in A \setminus A'} \textrm{d}_\infty (I_a, \Delta)^p +\sum_{b\in B\setminus B'} \textrm{d}_\infty (I_b, \Delta)^p
\] 
yields the \emph{$p$th-Wasserstein distance} between $\mathscr{B}_1$, $\mathscr{B}_2$:
\[
\textrm{d}_p(\mathscr{B}_1, \mathscr{B}_2) = (\min_\theta P_p(\theta))^{\frac{1}{p}} .
\]

\section{Max-Plus Polynomials on the Barcode Space}\label{sec:maxnplus} 
In this section we find all max-plus polynomials that we can use as coordinates on the barcode space and prove that they are stable with respect to the bottleneck and Wasserstein distances. 

The first step is to identify 2-symmetric max-plus polynomials on the image of $B_n \to B$. By abuse of notation we denote it simply by $B_n$. It is the quotient of the following equivalence relation: two multisets of $n$ intervals each, 
\[
I= \{(x_1,d_1), (x_2, d_2), \ldots, (x_n, d_n)\} \textrm{ and }J  = \{(x_1, d_1), (x_2, d_2), \ldots, (x_n, d_n)\},
\]
are equivalent if subsets $A, B \subseteq\{1,\ldots,n\}$ exist such that there is an equality of multisets $I\setminus\{(x_\alpha,0)\,:\,\alpha\in A\} = J\setminus\{(x_\beta,0)\,:\,\beta\in B\}$.

If $\mathcal{W}_j \subseteq ([0, \infty) \times [0, \infty))^n$ is the subset of $n$-tuples of pairs $(x_1, d_1, x_2, d_2, \ldots, x_n, d_n)$, with 0 persistence, i.e.\ $d_j = 0$, then these functions are precisely the 2-symmetric max-plus polynomials whose restriction to $\mathcal{W}_j$ is independent of $x_j$ for all $j$. 

\begin{lemma}\label{minprebar}
Let the minimal representation of a max-plus polynomial $p(x_1, d_1, \ldots, x_n, d_n)$ be
 \[
\displaystyle{\boxplus_{i = 1 \ldots, m} }a_0^i \odot x_1^{a_1^i} \odot d_1^{b_1^i}\odot  \ldots \odot x_n^{a_n^i}  \odot d_n^{b_n^i}.
 \]
Then $p$ restricted to $\mathcal{W}_j$ is independent of $x_j$ if and only if $a_j^i=0$ for all $i =1, \ldots, m$.
\end{lemma}
\begin{proof}
The direction $(\Leftarrow)$ follows immediately. We must show $(\Rightarrow)$. Choose $j$ and assume that $p$ restricted to $\mathcal{W}_{j}$ is independent of $x_{j}$.
Suppose not all $a_{j}^i$ are $0$. Let $i_0$ be such that 
\[
a_{j}^{i_0}= \displaystyle{ \max_{i=1, \ldots, m}} a_{j}^i .
\]
If this maximum is attained in more than one value, we choose the $i_0$ for which $a_0^i$ is the biggest.  
Observe that $p(0, \ldots, 0, x_{j}, 0, \ldots, 0) =  \displaystyle{\boxplus_{i = 1 \ldots, m} }a_0^i \odot x_j^{a_j^i}$. If $a^{i_0}_{j}>0$, then $p(0, \ldots, 0, x_{j}, 0, \ldots, 0) = a_0^{i_0}\odot  x_{j}^{a_{j}^{i_0}}$ for all $x_{j} >\displaystyle{ \max_{\{ i\,|\, a_0^{i_0} \neq  a_0^{i} \}} }\frac{ a_0^i -  a_0^{i_0}}{ a_0^{i_0} -  a_0^{i} }$. Here we take the maximum over $i$ for which $a_0^{i_0} \neq  a_0^{i}$. For indices $i$ when $a_0^{i_0} =  a_0^{i}$,  $a_0^{i_0} + a_j^{i_0}x_j \geq a_0^{i} + a_j^{i} x_j$ for our choice of $i_0$. This shows that for $a^{i_0}_{j}>0$ the max-plus polynomial $p(x_1, d_1, \ldots, x_n, d_n)$ depends on $x_{j}$. By assumption $a^{i_0}_{j} \geq 0$. The only way the expression does not depend on $x_{j}$ is if $a^{i_0}_{j} =0$. 
\end{proof}
\begin{corollary}
The subsemiring of max-plus polynomials whose restriction to $\mathcal{W}_i$ is independent of $x_i$ for all $i$ contains precisely the max-plus polynomials of the form
 \[
\boxplus_{i = 1 \ldots, m} a_0^i \odot d_1^{b_1^i}\odot  \ldots  \odot d_n^{b_n^i}.
 \]
\end{corollary}
We denote this semiring by $D_n$.

\begin{proposition}\label{generator}
Let $D_n^{S_n}$ denote the subring of elements of $D_n$ which are invariant under the action of $S_n$. Then $\sigma_{(0,1)}, \sigma_{(0,1)^2}, \ldots, \sigma_{(0,1)^n}$ generate $D_n^{S_n}$, in the sense that any element of $D_n^{S_n}$ is of the form
\[
\boxplus_{i = 1 \ldots, m} a_0^i \odot \sigma_{(0,1)}^{b_1^i}\odot  \ldots  \odot \sigma_{(0,1)^n}^{b_n^i},
\]
where $a_0^i\in \RR$ and all $b_j^i$ nonnegative integers.
\end{proposition}
\begin{proof}
According to Lemma~\ref{minprebar} the 2-symmetric max-plus polynomials on $B_n$ are precisely symmetric max-plus polynomials in variables $d_1, \ldots, d_n$. We can therefore apply Theorem~\ref{fundamentalminplus} with $\sigma_{(0,1)^k}$ playing the same role of $\sigma_k$. 
\end{proof}

Now that we have identified functions for each $B_n$ separately, we must assemble them to get functions on the barcode space. When $n \geq m$, the natural inclusion
\[
\begin{array}{rcl}
B_m &\to& B_n \\
\{(x_1, d_1), \ldots, (x_m, d_m)\} &\mapsto & \{(x_1, d_1), \ldots, (x_m, d_m), (0, 0), \ldots, (0,0)\} 
\end{array}
\]
induces $j_{n,m} \colon D_n \to D_m$, defined by 
\[
j_{n,m}(f)((x_1, d_1), \ldots, (x_m, d_m))  = f ((x_1, d_1), \ldots, (x_m, d_m), (0, 0), \ldots, (0,0)),
\]
The map $j_{n, m}$ is $S_m$-equivariant ($S_m$ acts by permuting the first $m$ pairs of variables). It follows that we may construct composites
\[
i_n^{m} \colon  D_n^{S_n}  \hookrightarrow D_n^{S_m} \xrightarrow{i_{n, m}^{S_m}} D_m^{S_m}
\]
and an inverse system
\[
\ldots \xrightarrow{i_{n}^{n+1}}  D_n^{S_n}  \xrightarrow{i_{n-1}^{n}}  D_{n-1}^{S_{n-1}}  \xrightarrow{i_{n-2}^{n-1}} \ldots \xrightarrow{\, \,i_{1}^{2}\,\,}  D_{1}^{S_{1}}. 
\]
Observe that 
\[
i_{n-1}^n (\sigma_{(0,1)^k}) = i_{n-1}^n (\sigma_{{\scriptsize \underbrace{(0,1),...,(0,1)}_\text{k}}}) = \sigma_{(0,1)^k}\, \,\textrm{ for }\,  k\neq n \quad \textrm{and} \quad i_{n-1}^{n} (\sigma_{(0,1)^{n}}) = \sigma_{(0,1)^{n-1}}. 
\]
Therefore $i_{n-1}^{n}$ are surjections for all positive integers $n$.
We do not wish to include functions with infinitely many variables, such as $\max_{i\in \NN} x_i$,  and for this reason we take a filtered inverse limit of these objects instead of the inverse limit. 
The total degree is the filter we use. Recall that $\operatorname{Deg} p$ of a max-plus polynomial
\[
p(x_1, x_2, \ldots, x_n) = a_1\odot x_1^{i_1^1} x_2^{i_2^1} \ldots x_n^{i_n^1} \boxplus a_2\odot x_1^{i_1^2} x_2^{i_2^2} \ldots x_n^{i_n^2}\boxplus \ldots \boxplus a_m\odot x_1^{i_1^m} x_2^{i_2^m} \ldots x_n^{i_n^m}
\] 
is $\operatorname{max}_{1\leq j \leq m}(i_1^j+ i_2^j + \ldots + i_n^j)$. Let
 \[
 {}_kD_n = \{ p\in D_n\,|\, \textrm{Deg}\,p \leq k \}
 \]
 Map $i_{n-1}^{n}$ induces ${}_ki_{n-1}^{n}\colon {}_kD_n^{S_n}  \xrightarrow{{}_ki_{n-1}^{n}}  {}_kD_{n-1}^{S_{n-1}}$. We denote the inverse limit of this system by $\mathscr{D}^k$. The space of max-plus polynomials on the barcode space, $\mathscr{D}$, is precisely $ \displaystyle{  \bigcup_{k=1}^\infty \mathscr{D}^k}$. 
 
\begin{definition}
A semiring $(\mathscr{R}, +, \cdot)$ is called \emph{filtered} if there exists such a family of subsemirings $\{\mathscr{R}_d \}_{d\in \NN}$ of $(\mathscr{R}, +, \cdot)$ for operation $+$ that
\begin{itemize}
\item
$\mathscr{R}_d \subset \mathscr{R}_{d'}$ for $d\leq d'$,
\item
$\mathscr{R} = \bigcup_d \mathscr{R}_d$, 
\item
$\mathscr{R}_d \cdot \mathscr{R}_{d'} \subset \mathscr{R}_{d+d'}$ for all $d, d' \in \NN$.
\end{itemize}
\end{definition}
\begin{theorem}\label{maxplusgen}
Max-plus polynomials on the barcode space, $\mathscr{D}$, have the structure of a filtered semiring. They are generated by elements of the form $\sigma_{(0,1)^n}$, where $n$ is a positive integer.
\end{theorem}

\section{Stability of Max-Plus Polynomials}
Stability is the key property that coordinate functions should satisfy. In this section we prove that the functions from $\mathscr{D}$ are stable with respect to the bottleneck and Wasserstein distances.

\begin{theorem}[Bottleneck stability of max-plus polynomials]\label{maxplus}
Let $\mathscr{D}$ be the filtered semiring of max-plus polynomials (see Theorem~\ref{maxplusgen}). If $F \in \mathscr{D}$, then a constant $C$ exists such that
\[
|F(\mathscr{B}_1) - F(\mathscr{B}_2)| \leq C \textrm{d}_\infty(\mathscr{B}_1, \mathscr{B}_2)
\]
for any pair of barcodes $\mathscr{B}_1$ and  $\mathscr{B}_2$.
\end{theorem}
\begin{lemma}\label{lemmamaxplus}
For any pair of barcodes $\mathscr{B}_1$ and  $\mathscr{B}_2$ and any $n\in\NN$, the difference between the total length of the longest $n$ bars in $\mathscr{B}_1$ and $\mathscr{B}_2$ can be bounded from above by $2n \textrm{d}_\infty(\mathscr{B}_1, \mathscr{B}_2)$:
\[
|\sigma_{(0,1)^n}(\mathscr{B}_1) - \sigma_{(0,1)^n}(\mathscr{B}_2)| \leq 2n \textrm{d}_\infty(\mathscr{B}_1, \mathscr{B}_2).
\]
\end{lemma}
\begin{proof}
Let $\mathscr{B}_1 = \{ (x_1, d_1), \ldots, (x_{l_1}, d_{l_1}) \}$ and  $\mathscr{B}_2 = \{(x_1', d_1'), \ldots, (x_{l_2}', d_{l_2}')\}$ be such that ${\mathscr{B}_1 \neq \mathscr{B}_2}$ and $d_1 \geq d_2 \geq \ldots \geq d_{l_1} \geq 0$. 

Without loss of generality assume that $\sigma_{(0,1)^n} (\mathscr{B}_1) \geq \sigma_{(0,1)^n} (\mathscr{B}_2)$. If $n > l_1$ or $n > l_2$, we add 0 length intervals to $\mathscr{B}_1$, $\mathscr{B}_2$ to achieve that their length is $n$.

Let $\theta$ be a bijection where the penalty is minimal, i.e.\ where ${P_\infty(\theta) = \textrm{d}_\infty(\mathscr{B}_1, \mathscr{B}_2)}$. Assume that $\theta$ matches $(x_1, d_1)$ with $(x_1', d_1')$, $(x_2, d_2)$ with $(x_2', d_2')$, \ldots, $(x_n, d_n)$ with $(x_n', d_n')$ (some of these intervals might be 0 length intervals). Of course, we might have to relabel bars in $\mathscr{B}_2$ for this to hold.  
Note that for those $i$, for which either $d_i$ or $d_{i}'$ equals 0, we automatically have
\[
|\frac{d_i-d_i'}{2}| \leq \textrm{d}_\infty(\mathscr{B}_1, \mathscr{B}_2).
\]
For all other $1 \leq i \leq n$ in this matching,
\begin{equation}\label{eq1}
|\frac{d_i-d_i'}{2}| \leq \max (|x_i-x_i'|, |d_i-d_i'+x_i-x_i'|) \leq \textrm{d}_\infty(\mathscr{B}_1, \mathscr{B}_2).
\end{equation}
By the definition of a minimal matching $\displaystyle \max_{i=1, \ldots, m}(|x_i-x_i'|, |d_i-d_i'+x_i-x_i'|) \leq \textrm{d}_\infty(\mathscr{B}_1, \mathscr{B}_2)$. So we must only prove the first inequality. Notice that if $|\frac{d_i-d_i'}{2}| \leq |x_i-x_i'|$, this follows automatically. If $|\frac{d_i-d_i'}{2}| >|x_i-x_i'|$, then
\[
|\frac{d_i-d_i'}{2}| \leq |d_i-d_i'+x_i-x_i'|,
\] 
proving Inequality~\ref{eq1}.

Then 
\[
\begin{array}{lcl}
n\textrm{d}_\infty(\mathscr{B}_1, \mathscr{B}_2) & \geq & \sum_{i=1}^n\frac{(d_i - d_i')}{2} \\
 &=& \frac{1}{2} (\sum_{i=1}^n d_i - \sum_{i=1}^n d_i') \\
& = & \frac{1}{2} (\sigma_{(0,1)^n} (\mathscr{B}_1)-\sum_{i=1}^n d_i')\\
& \geq & \frac{1}{2} (\sigma_{(0,1)^n} (\mathscr{B}_1)-\sigma_{(0,1)^n} (\mathscr{B}_2)).\\
\end{array}
\] 
The last inequality holds since $\sum_{i=1}^n d_i' \leq \sigma_{(0,1)^n} (\mathscr{B}_2)$. Also note that we chose $d_1, \ldots, d_n$ in a way that $\sigma_{(0,1)^n} (\mathscr{B}_1)= \sum_{i=1}^n d_i$.

We deduce that 
\[
|\sigma_{(0,1)^n} (\mathscr{B}_1)-\sigma_{(0,1)^n} (\mathscr{B}_2)| \leq 2n \textrm{d}_\infty(\mathscr{B}_1, \mathscr{B}_2),
\]
proving that $\sigma_{(0,1)^n}$ is Lipschitz with constant $2n$.
\end{proof}

\begin{proof}[Proof of Theorem~\ref{maxplus}]
Suppose $F_1$ and $F_2$ are such that $C_1$ and $C_2$ exist such that
\[
|F_1(\mathscr{B}_1) - F_1(\mathscr{B}_2)| \leq C_1 \textrm{d}_\infty(\mathscr{B}_1, \mathscr{B}_2)
\]
and
\[
|F_2(\mathscr{B}_1) - F_2(\mathscr{B}_2)| \leq C_2 \textrm{d}_\infty(\mathscr{B}_1, \mathscr{B}_2)
\]
for any any pair of barcodes $\mathscr{B}_1$ and  $\mathscr{B}_2$. 

Let $H = F_1 + F_2$. Then
\[
\begin{array}{lcl}
|H(\mathscr{B}_1) - H(\mathscr{B}_2)| &= & |F_1(\mathscr{B}_1) + F_2(\mathscr{B}_1) - F_1(\mathscr{B}_2) - F_2(\mathscr{B}_2)|  \\
& \leq & |F_1(\mathscr{B}_1) - F_1(\mathscr{B}_2)| + |F_2(\mathscr{B}_1) - F_2(\mathscr{B}_2)|  \\
& \leq & C_1 \textrm{d}_\infty(\mathscr{B}_1, \mathscr{B}_2) + C_2 \textrm{d}_\infty(\mathscr{B}_1, \mathscr{B}_2)  \\
&\leq & (C_1 + C_2) \textrm{d}_\infty(\mathscr{B}_1, \mathscr{B}_2).
\end{array}
\]

Let $H = \max (F_1, F_2)$. Then
\[
F_1(\mathscr{B}_2) \leq F_1(\mathscr{B}_1)+ |F_1(\mathscr{B}_2)-F_1(\mathscr{B}_1)| \leq H(\mathscr{B}_1)+  |F_1(\mathscr{B}_2)-F_1(\mathscr{B}_1)|,
\]
and similarly $F_2(\mathscr{B}_2) \leq  H(\mathscr{B}_1) + |F_2(\mathscr{B}_2)-F_2(\mathscr{B}_1)|$. It follows that
\[
H(\mathscr{B}_2) \leq H(\mathscr{B}_1) +\max (|F_1(\mathscr{B}_2)-F_1(\mathscr{B}_1)|, |F_2(\mathscr{B}_2)-F_2(\mathscr{B}_1)|), 
\]
and by symmetry we conclude that
\[
|H(\mathscr{B}_1) - H(\mathscr{B}_2)| \leq \max(C_1, C_2) \textrm{d}_\infty(\mathscr{B}_1, \mathscr{B}_2).
\]

Any function $F$ from the filtered semiring of max-plus polynomials $\mathscr{D}$ is generated by taking maxima and sums of $\sigma_{(0,1)^n}$ and constants. Since stability is preserved under these two operations and since $\sigma_{(0,1)^n}$ are stable according to Lemma~\ref{lemmamaxplus}, $F$ is also stable.
\end{proof}

\begin{theorem}[Wasserstein stability of max-plus polynomials]\label{wasmaxplus}
Let $\mathscr{D}$ be the filtered semiring of max-plus polynomials. For $F \in \mathscr{D}$ and $p \geq 1$, a constant $C$ exists such that
\[
|F(\mathscr{B}_1) - F(\mathscr{B}_2)| \leq C\, \textrm{d}_p(\mathscr{B}_1, \mathscr{B}_2)
\]
for any pair of barcodes $\mathscr{B}_1$ and  $\mathscr{B}_2$.
\end{theorem}
\begin{proof}
Let $\mathscr{B}_1 = \{ (x_1, d_1), \ldots, (x_{l_1}, d_{l_1}) \}$ and  $\mathscr{B}_2 = \{(x_1', d_1'), \ldots, (x_{l_2}', d_{l_2}')\}$ be such that ${\mathscr{B}_1 \neq \mathscr{B}_2}$ and $d_1 \geq d_2 \geq \ldots \geq d_{l_1} \geq 0$. 

Without loss of generality assume that $\sigma_{(0,1)^n} (\mathscr{B}_1) \geq \sigma_{(0,1)^n} (\mathscr{B}_2)$. If $n > l_1$ or $n > l_2$, we add 0 length intervals to $\mathscr{B}_1$, $\mathscr{B}_2$ to achieve that their length is $n$.

Let $\theta$ be a bijection where the penalty is minimal, i.e.\ where ${P_p(\theta) = \textrm{d}_p(\mathscr{B}_1, \mathscr{B}_2)}$. Assume that $\theta$ matches $(x_1, d_1)$ with $(x_1', d_1')$, $(x_2, d_2)$ with $(x_2', d_2')$, \ldots, $(x_n, d_n)$ with $(x_n', d_n')$ (some of these intervals might be 0 length intervals).  

Note that for those $i$, for which either $d_i$ or $d_{i}'$ equals 0, we automatically have
\[
|\frac{d_i-d_i'}{2}| \leq \textrm{d}_\infty(\mathscr{B}_1, \mathscr{B}_2).
\]
For all other $i$ in this matching,
\[
|\frac{d_i-d_i'}{2}|^p \leq (\max |x_i-x_i'|, |d_i-d_i'+x_i-x_i'|)^p
\]
since $x \mapsto x^p$ is increasing for $x >0$ (note that $p \geq 1$).
Then 
\[
\begin{array}{lcl}
 (\sigma_{(0,1)^n} (\mathscr{B}_1)-\sigma_{(0,1)^n} (\mathscr{B}_2))^p & \leq & (\sigma_{(0,1)^n} (\mathscr{B}_1)-\sum_{i=1}^n d_i')^p \\
 &=& (\sum_{i=1}^n d_i - \sum_{i=1}^n d_i')^p\\
  &\leq & 2^p (\sum_{i=1}^n |\frac{d_i-d_i'}{2}|)^p \\
& \leq & 2^p (n)^{p-1} (\sum_{i=1}^n |\frac{d_i-d_i'}{2}|^p) \\
&\leq & 2^p (n)^{p-1} P_p(\theta)^p\\
&=& 2^p n^{p-1}\textrm{d}_p(\mathscr{B}_1, \mathscr{B}_2)^p.\\
\end{array}
\] 
The first inequality holds since $\sum_{i=1}^n d_i' \leq \sigma_{(0,1)^n} (\mathscr{B}_2)$. Also note that we chose $d_1, \ldots, d_n$ in a way that $\sigma_{(0,1)^n} (\mathscr{B}_1)= \sum_{i=1}^n d_i$.
To bound $\sum_{i=1}^n |\frac{d_i-d_i'}{2}|^p$ we use H{\H o}lder's inequality.

It follows from here that
\[
|\sigma_{(0,1)^n} (\mathscr{B}_1)-\sigma_{(0,1)^n} (\mathscr{B}_2)| \leq 2 n^{\frac{p-1}{p}}\textrm{d}_p(\mathscr{B}_1, \mathscr{B}_2).
\]
The statement of the theorem now follows from the same argument  as in the proof of Theorem~\ref{maxplusgen}.
\end{proof}

\section{Tropical Rational Functions on the Barcode Space}\label{sec:troprat} 
While the the functions belonging to $\mathscr{D}$ are stable and can be used to assign vectors to barcodes, they do not separate points in the barcode space, because they are composed by taking sums and maxima of lengths of intervals and constants. One example is $\{ (1, 2), (2, 2)\}$ and $\{ (2, 2), (3, 2) \}$. We can easily convince ourselves of this by evaluating $\sigma_{(0,1)^n}$ on these barcodes. 

Because there simply are not enough functions among max-plus polynomials to separate points, we expand the set of functions we observe to all tropical rational functions. Let 
\[
((x_1, d_1), \ldots, (x_n, d_n)), ((x_1', d_1'), \ldots, (x_n', d_n'))\in  [0, \infty)^{2n}.
\]
Without loss of generality we assume that they are lexicographically ordered.

The tropical rational functions that respect the equivalence classes of $B_n$, must respect the following equivalence relation $\sim$ on $[0, \infty)^{2n}$:
\[
((x_1, d_1), \ldots, (x_n, d_n)) \sim ((x_1', d_1'), \ldots, (x_n', d_n')) \Leftrightarrow (\forall i:\, (d_i = d_i' \land (x_i = x_i' \lor d_i =0))).
\]
We denote the semiring of such functions by $R_n$.
\begin{theorem}\label{finsub}
No finite subset of $R_n$ exists which separates nonequivalent points in $B_n$.
\end{theorem}
\begin{proof}
Assume $\{f_1, \ldots, f_m\}\in R_n$ separates nonequivalent points in $B_n$. Let $\vec{x} = (x_1, d_1, \ldots, x_n, d_n)$ and $\vec{x}' = (x_1', d_1', \ldots, x_n', d_n')$. We define
\[
g(\vec{x}, \vec{x}') = \max \{ | f_1(\vec{x}) - f_1(\vec{x}')|, \ldots, | f_m(\vec{x}) - f_m(\vec{x}')| \}.
\] 
The function $g$ is the $L^{\infty}$-distance between vectors $(f_1(\vec{x}), \ldots, f_m(\vec{x}))$ and $(f_1(\vec{x}'), \ldots, f_m(\vec{x}'))$. Thus $g(\vec{x}, \vec{x}') = 0$ if and only if $\vec{x}$ and $\vec{x}'$ are equivalent points.

Since $|x| = \max(x, -x)$, $g$ is a tropical rational function and as demonstrated in Subsection~\ref{ratfun} we can write $g(\vec{x},\vec{x}')$ as
\begin{multline}
\max_{i=1, \ldots, l_1} (\sum_{k=1}^n (a_{k,i} x_k + b_{k,i} d_k) + \sum_{k=1}^n (a_{k,i}' x_k' + b_{k,i}' d_k') + c_{i}) \,-\\
\max_{j=1, \ldots, l_2} (\sum_{k=1}^n (s_{k,j} x_k + t_{k,j} d_k) + \sum_{k=1}^n (s_{k,j}' x_k' + t_{k,j}' d_k') + u_{j}).
\end{multline}
For any $x\geq 0$, define $\vec{p}_x = (x, 0, \ldots, x, 0)$. Since $\vec{p}_x $ and $\vec{p}_y$ are equivalent  for any $x, y \geq 0$ with respect to the relation defined before the statement of Theorem~\ref{finsub}, $g(\vec{p}_x, \vec{p}_y) = 0$ and consequently
\[
\max_{i=1, \ldots, l_1} (x \sum_{k=1}^n a_{k,i} + y \sum_{k=1}^n a_{k,i}'  + c_{i}) =
\max_{j=1, \ldots, l_2} (x \sum_{k=1}^n s_{k,j} + y \sum_{k=1}^n s_{k,j}' + u_{j}).
\]

Both $\displaystyle{\max_{i=1, \ldots, l_1} (x \sum_{k=1}^n a_{k,i} + y \sum_{k=1}^n a_{k,i}'  + c_{i})}$ and $\displaystyle{\max_{j=1, \ldots, l_2} (x \sum_{k=1}^n s_{k,j} + y \sum_{k=1}^n s_{k,j}' + u_{j})}$ are piecewise linear functions defined on $\RR^2$. Each function defines a decomposition of $\RR^2$ into maximal closed domains over which this function is linear on every domain (Lemma~\ref{decomposition}).  Let \[
D_i = \{(x, y) \in [0, \infty)^2\,|\, x \sum_{k=1}^n a_{k,i} + y \sum_{k=1}^n a_{k,i}'  + c_{i} > x \sum_{k=1}^n a_{k, j} + y \sum_{k=1}^n a_{k, j}'  + c_{j} \textrm{ for all } j\neq i \},
\] 
\[
E_i = \{(x, y) \in [0, \infty)^2\,|\, x \sum_{k=1}^n s_{k,i} + y \sum_{k=1}^n s_{k,i}' + u_{i} > x \sum_{k=1}^n s_{k,j} + y \sum_{k=1}^n s_{k,j}' + u_{j} \textrm{ for all } j\neq i \}.
\]
These sets are open.
 Two linear functions on an non-empty open set are the same if and only if their coefficients are the same. This implies that $\sum_{k=1}^n a_{k,i} = \sum_{k=1}^n s_{k,j}$, $\sum_{k=1}^n a_{k,i}' = \sum_{k=1}^n s_{k,j}' $  and $c_i = u_j$ for all $i$ and $j$ for which ${D_i \cap E_j \neq \emptyset}$.
 
We also define
\[
S_i = \{(\vec{x}, \vec{x'}) \in [0, \infty)^{2n}\times [0, \infty)^{2n}\,|\, \sum_{k=1}^n (a_{k,i} x_k + b_{k,i} d_k) + \sum_{k=1}^n (a_{k,i}' x_k' + b_{k,i}' d_k') + c_{i}
\]
\[> \sum_{k=1}^n (a_{k, j} x_k + b_{k, j} d_k) + \sum_{k=1}^n (a_{k,j}' x_k' + b_{k, j}' d_k') + c_{j}\textrm{ for all } j\neq i \}
\] 
and
\[
T_i =  \{(\vec{x}, \vec{x'}) \in [0, \infty)^{2n} \times [0, \infty)^{2n}\,|\, \sum_{k=1}^n (s_{k,i} x_k + t_{k,i} d_k) + \sum_{k=1}^n (s_{k,i}' x_k' + t_{k,i}' d_k') + u_{i}
\]
\[> \sum_{k=1}^n (s_{k,j} x_k + t_{k,j} d_k) + \sum_{k=1}^n (s_{k,j}' x_k' + t_{k,j}' d_k') + u_{j}\textrm{ for all } j\neq i \}.
\] 
$\bigcup S_i$ and $\bigcup T_i$ are open and dense in $[0, \infty)^{2n} \times [0, \infty)^{2n}$. In particular 
\[
[0, \infty)^{2n}\times [0, \infty)^{2n} = \bigcup_{i, j} \overline{S_i \cap T_j}.
\]
The closures $\overline{S_i \cap T_j}$ have piecewise linear boundaries by Lemma~\ref{decomposition} since $g$ is a tropical rational function. We claim that there exist indices $i$ and $j$ and $\epsilon, a >0$ such that the  points 
 \[
 \begin{array}{lcl}
A &:=& ((0, 0, \ldots, 0, 0),(0, 0, \ldots, 0, 0)),\\
 B &:=&((0, \epsilon, 0, \epsilon, \ldots, 0, \epsilon),(0, \epsilon, 0, \epsilon, \ldots, 0, \epsilon)),\\
 C &:=& ((0, \epsilon, 0, \epsilon, \ldots, 0, \epsilon),(a, \epsilon, a, \epsilon, \ldots, a, \epsilon))
\end{array} 
\] 
belong to $\overline{S_i \cap T_j}$. We will show that $g(C) = 0$ although $(0, \epsilon, 0, \epsilon, \ldots, 0, \epsilon)$ and $(a, \epsilon, a, \epsilon, \ldots, a, \epsilon)$ are not equivalent points. This will lead to a contradiction.
 
 The point $A$, together with vectors 
 \[
 ((0, 1, 0, 1, \ldots, 0, 1),(0, 1, 0, 1, \ldots, 0, 1)) \textrm{ and }((0, 0, \ldots, 0, 0),(1, 0, \ldots, 1, 0)),
 \] 
 determines a plane; denote its intersection with $[0, \infty)^{2n} \times [0, \infty)^{2n}$ by $P$. The decomposition of $[0, \infty)^{2n} \times [0, \infty)^{2n}$ on $\overline{S_i \cap T_j}$ determines a decomposition of $P$ on sets $\overline{S_i \cap T_j} \cap P$ -- their boundaries (considered within $P$) are then also piecewise linear.
 
We set
 \[
 \vec{v} := ((0, 1, 0, 1, \ldots, 0, 1),(0, 1, 0, 1, \ldots, 0, 1))
 \]
 and
 \[
 \vec w=((0, 0, \ldots, 0, 0),(1, 0, \ldots, 1, 0)).
 \]
 If $A$ is in the interior (relative to $P$) of some $\overline{S_i \cap T_j} \cap P$, then we can clearly find the required $\epsilon$ and $a$. If $A$ is in a boundary, consider what lies in the direction $\vec{v}$. If it is the interior of some $\overline{S_i \cap T_j} \cap P$, then $A$ is in $\overline{S_i \cap T_j} \cap P$, and we can again find suitable $\epsilon$ and $a$. Suppose now that $A$ lies in a boundary $\ell$ which continues in the direction of $\vec{v}$. Let $\epsilon$ be small enough that $\ell$ does not yet end at $\epsilon \vec{v}$, and such that the line segment open at $A$ joining $A$ and $A+\epsilon\vec{v}$ does not intersect any boundaries other than $\ell$; we then let $B=A+\epsilon\vec{v}$. Since this boundary did not yet end at $B$, the vector $\vec w$ with the origin point in $B$ must point into the interior of some $\overline{S_i \cap T_j} \cap P$; choose a small enough $a$ so that $C=B+a\vec{w}$, and indeed the line segment joining $B$ and $C$ lies within $\overline{S_i\cap T_j}\cap P$. 
 
Now we calculate
\[
\begin{array}{lcl}
0 &=& g((0, \epsilon, \ldots, 0, \epsilon), (0, \epsilon, \ldots, 0, \epsilon)) \\
&=&  \epsilon \sum_{k=1}^n (b_{k,i} + b_{k,i}') + c_{i} - (\epsilon \sum_{k=1}^n (t_{k,j} + t_{k,j}') + u_{j})\\
&=&   a \sum_{k=1}^n a_{k,i}' + \epsilon \sum_{k=1}^n (b_{k,i} + b_{k,i}') + c_{i} -\\
& & \,\,\,( a \sum_{k=1}^n s_{k,j}' + \epsilon \sum_{k=1}^n (t_{k,j} + t_{k,j}') + u_{j})\\
 &=& g((0 , \epsilon, \ldots, 0, \epsilon), (a, \epsilon, \ldots, a, \epsilon)) \\
  &\neq& 0,\\
\end{array}
\]
which is a contradiction.
\end{proof}

Theorem~\ref{finsub} states that no finite subset of symmetric min-plus, max-plus or tropical rational functions exists that separates barcodes. In this section we identify a countable set of tropical rational functions on the barcode space that does.

\begin{theorem}\label{rationalcoord}
Let $\{ \sigma_{(e_{1, 1}, e_{1, 2}), \ldots, (e_{n, 1}, e_{n, 2}) }\}$ be the set of elementary 2-symmetric max-plus polynomials from Definition~\ref{2symel}.
Functions, defined by
\[
E_{m, (e_{1, 1}, e_{1, 2}), \ldots, (e_{n, 1}, e_{n, 2}) }(x_1, d_1, \ldots, x_n, d_n) :=\sigma_{(e_{1, 1}, e_{1, 2}), \ldots, (e_{n, 1}, e_{n, 2}) } (x_1 \oplus d_1^m, d_1, \ldots, x_n \oplus d_n^m, d_n),
\] 
for $m\in \NN$ are contained in $R_n$. Furthermore, they separate nonequivalent points in $B_n$.
\end{theorem} 

\begin{proof}
Restricted to $d_i=0$ for $i \in \NN_{\leq n}$, expressions $x_i \oplus d_i^m$ are 0 and therefore independent of $x_i$ and consequently so are their post-compositions with $e_{(e_{1, 1}, e_{1, 2}), \ldots, (e_{n, 1}, e_{n, 2}) }$. This implies that $E_{m, (e_{1, 1}, e_{1, 2}), \ldots, (e_{n, 1}, e_{n, 2}) }(x_1, d_1, \ldots, x_n, d_n)$ is contained in $R_n$. 

We must show that if  $(x_1, d_1, \ldots, x_n, d_n)$ and $(x_1', d_1', \ldots, x_n', d_n')$ are not equivalent in $B_n$, we can find such $E_{m, (e_{1, 1}, e_{1, 2}), \ldots, (e_{n, 1}, e_{n, 2}) }$ that
\[
E_{m, (e_{1, 1}, e_{1, 2}), \ldots, (e_{n, 1}, e_{n, 2}) }(x_1, d_1, \ldots, x_n, d_n) \neq E_{m, (e_{1, 1}, e_{1, 2}), \ldots, (e_{n, 1}, e_{n, 2}) }(x_1', d_1', \ldots, x_n', d_n').
\]
Let $(x_1, d_1, \ldots, x_n, d_n)$ and $(x_1', d_1', \ldots, x_n', d_n')$ be nonequivalent. Without loss of generality assume that $d_1 \leq \ldots \leq d_n$ and $d_1' \leq \ldots \leq d_n'$. 

Some of the $d$'s, say $d_1, \ldots, d_{k-1} = 0$ can be 0 (if $k=1$ none of $d$'s is 0). The point $(x_1, 0, \ldots, x_{k-1}, 0, x_{k}, d_k, \ldots x_n, d_n)$ is equivalent to $(0, 0, \ldots, 0, 0, x_{k}, d_k, \ldots x_n, d_n)$ and consequently
\begin{multline}
E_{m, (e_{1, 1}, e_{1, 2}), \ldots, (e_{n, 1}, e_{n, 2}) }(x_1, 0, \ldots, x_{k-1}, 0, x_{k}, d_k, \ldots x_n, d_n) =\\
 E_{m, (e_{1, 1}, e_{1, 2}), \ldots, (e_{n, 1}, e_{n, 2}) }(0, 0, \ldots, 0, 0, x_{k}, d_k, \ldots x_n, d_n)
\end{multline}
for all $m$ and $(e_{1, 1}, e_{1, 2}), \ldots, (e_{n, 1}, e_{n, 2})$. Similarly, if $d_1', \ldots, d_{l-1}' = 0$, then 
\[
(x_1', 0, \ldots, x_{l-1}', 0, x_{l}', d_l', \ldots x_n', d_n') \sim (0, 0, \ldots, 0, 0, x_{l}', d_l', \ldots x_n', d_n')
\]
and consequently
\begin{multline}
E_{m, (e_{1, 1}, e_{1, 2}), \ldots, (e_{n, 1}, e_{n, 2}) }(x_1', 0, \ldots, x_{l-1}', 0, x_{l}', d_l', \ldots x_n', d_n')=\\ E_{m, (e_{1, 1}, e_{1, 2}), \ldots, (e_{n, 1}, e_{n, 2}) }(0, 0, \ldots, 0, 0, x_{l}', d_l', \ldots x_n', d_n')
\end{multline}
for all $m$ and $(e_{1, 1}, e_{1, 2}), \ldots, (e_{n, 1}, e_{n, 2})$. 

Choose $m \in \NN$ such that 
\[
m  > \max (\max_{k \leq i \leq n} \frac{x_i}{d_i},  \max_{l \leq i \leq n} \frac{x_i'}{d_i'}).
\]
For this $m$,
\[
(x_1 \oplus d_1^m, d_1, \ldots, x_n \oplus d_n^m, d_n) = (0, 0, \ldots, 0, 0, x_{k}, d_k, \ldots x_n, d_n) 
\]
and
\[
(x_1' \oplus d_1^{'m}, d_1', \ldots, x_n' \oplus d_n^{'m}, d_n) = (0, 0, \ldots, 0, 0, x_{l}', d_l', \ldots x_n', d_n')
\]
Proposition~\ref{sepbar} guarantees existence of $e \in \{ \sigma_{(e_{1, 1}, e_{1, 2}), \ldots, (e_{n, 1}, e_{n, 2}) }\}$ such that 
\[
e(0, 0, \ldots, 0, 0, x_{k}, d_k, \ldots x_n, d_n) \neq e(0, 0, \ldots, 0, 0, x_{l}', d_l', \ldots x_n', d_n').
\]
Therefore we see that for this choice of $m$ and this $e$,
\[
E_{m, e}(x_1, d_1, \ldots x_n, d_n) \neq E_{m, e}(x_1', d_1', \ldots, x_n', d_n')
\]
and we are done.
\end{proof}

It is hard to characterize \emph{all} tropical rational functions on $B_n$, so we work with a subsemiring of functions obtained by taking maxima, adding and substracting functions from $\{E_{m, (e_{1, 1}, e_{1, 2}), \ldots, (e_{n, 1}, e_{n, 2})}\}$. We denote this subsemiring by $G_n$ or $G_n^{S_n}$ when we wish to stress that all the functions contained in it are symmetric.
We have restriction maps $i_{n,m} \colon G_n \to G_m$, when $n \geq m$, induced by 
\[
i_{n,m}(f)(x_1, d_1, \ldots, x_m, d_m, \ldots, x_n, d_n) = f(x_1, d_1, \ldots, x_m, d_m, 0,0, \ldots, 0, 0),
\]
The map $i_{n, m}$ is $S_m$-equivariant, where $S_m$ acts by permuting the first $m$ pairs of variables. 

Maps $i_{n, n-1}$ transform the generators of $G_n$ as follows:
\[
E_{m, (0, 0)^j (1,0)^k (0, 1)^l (1, 1)^{p}} \mapsto \left\{ 
\begin{array}{ll}
E_{m, (0, 0)^{j-1} (1,0)^k (0, 1)^l (1, 1)^{p}} &\textrm{if }j \neq 0 \\
E_{m, (1, 1)^{n-1}} &\textrm{if }j= 0, k=0, l=0 \\
E_{m, (0, 1)^{l-1} (1, 1)^{p}} &\textrm{if }j= 0, k=0, l\geq 1 \\
E_{m, (1,0)^{k-1} (1, 1)^{p}} &\textrm{if }j= 0, k\geq , l=0  \\
E_{m, (1,0)^{k-1} (0, 1)^l (1, 1)^{p}} \boxplus E_{m, (1,0)^{k} (0, 1)^{l-1} (1, 1)^{p}}  & \textrm{if }j= 0, k\geq , l\geq1  \\
\end{array}\right .
\]
Here $p = n-l-k-j$. Therefore we $i_{n, n-1}$ is a surjection from $G_n$ to $G_{n-1}$ and we may construct composites
\[
i_n^{n-1} \colon G_n^{S_n} \hookrightarrow G_n^{S_{n-1}}\xrightarrow{i_{n, n-1}^{S_{n-1}}} G_{n-1}^{S_{n-1}}.
\]
We cannot proceed as we did in the case of max-plus polynomials, since we cannot define a degree of a tropical rational expression. However, recall that according to Section~\ref{back} we can write any $r\in G_n$ as
\begin{multline}
\max_{i=1, \ldots, l_1} (\sum_{k=1}^n (a_{k,i} x_k + b_{k,i} d_k) + \sum_{k=1}^n (a_{k,i}' x_k' + b_{k,i}' d_k') + c_{i}) \,-\\
\max_{j=1, \ldots, l_2} (\sum_{k=1}^n (s_{k,j} x_k + t_{k,j} d_k) + \sum_{k=1}^n (s_{k,j}' x_k' + t_{k,j}' d_k') + u_{j}).
\end{multline}
Now set
\[
{}_kG_n^{S_n} = \{ r\in G_n \,|\, r \sim p \oplus q^{-1}, p, q\textrm{ are max-plus polynomials with } \textrm{deg}\, p, \textrm{deg}\,q \leq k \}
\]
 Map $i_{n}^{n-1}$ induces ${}_ki_{n}^{n-1}\colon {}_kG_n^{S_n}  \xrightarrow{{}_ki_{n}^{n-1}}  {}_kG_{n-1}^{S_{n-1}}$. We denote the inverse limit of this system by $\mathscr{G}^k$. Let $\mathscr{G} = \cup_{k=1}^\infty \mathscr{G}^k$. 

\begin{theorem}\label{tropicalgen}
Tropical rational functions in $\mathscr{G}$ form a filtered semiring and they separate points in the barcode space. As a semiring $\mathscr{G}$ is generated by elements of the form $E_{m, (1,0)^k (0, 1)^l (1, 1)^{p}}$ where $k, l, p$ are nonnegative integers and $m$ is a positive integer.
\end{theorem}

\section{Stability of Tropical Rational Functions in $\mathscr{G}$}
In this subsection we prove that the rational functions that we identified are stable with respect to the bottleneck and Wasserstein distances.
\begin{theorem}[Bottleneck stability of functions in $\mathscr{G}$]\label{tropratstab}
If $F \in \mathscr{G}$, then a constant $C$ exists such that
\[
|F(\mathscr{B}_1) - F(\mathscr{B}_2)| \leq C \textrm{d}_\infty(\mathscr{B}_1, \mathscr{B}_2)
\]
for any pair of barcodes $\mathscr{B}_1$ and  $\mathscr{B}_2$.
\end{theorem}
\begin{lemma}\label{lemmam}
Let $m\in \NN$, $m_i= \min\{ x_i, md_i\}$ and $m_i'= \min\{ x_i', md_i'\}$. Then
\[
|m_i - m_i'| \leq 2m \max(|x_i-x_i'|, |d_i-d_i'+x_i-x_i'|).
\]
\end{lemma}
\begin{proof}
If $x_i \leq m d_i$ and $x_i' \leq m d_i'$, then
\[
|m_i - m_i'| = |x_i-x_i'|.
\] 
If $x_i \geq m d_i$ and $x_i' \geq m d_i'$, then
\[
|m_i - m_i'| = |m d_i-md_i'| = m|d_i-d_i'|.
\]
Let $x_i \leq m d_i$ and $x_i' > m d_i'$ (the case when $x_i > m d_i$ and  $x_i' \leq m d_i'$ is analogous). Since $0 \leq x_i \leq m d_i$, 
\[
-md_i' \leq x_i-md_i' \leq m(d_i-d_i').
\]
On the other hand $-x_i' < -m d_i' \leq 0$ and consequently
\[
x_i-x_i' < x_i-m d_i' \leq x_i.
\]
It follows that
\[
|x_i - m d_i'| \leq \max\{|x_i-x_i'|, m |d_i-d_i'|\}
\]
and consequently
\[
|m_i - m_i'| \leq \max\{|x_i-x_i'|, m |d_i-d_i'| \} \leq m\max\{|x_i-x_i'|, |d_i-d_i'| \}.
\]
By triangle inequality 
\[
|d_i-d_i'| \leq |d_i-d_i'+x_i-x_i'| + |x_i-x_i'| \leq 2  \max(|x_i-x_i'|, |d_i-d_i'+x_i-x_i'|). 
\]
Finally these two inequalities imply
\[ 
 \max(|x_i-x_i'|, |d_i-d_i'+x_i-x_i'|) \leq 2m \max(|x_i-x_i'|, |d_i-d_i'+x_i-x_i'|)
 \]
\end{proof}
\begin{proof}[Proof of Theorem~\ref{tropratstab}]
 Take $E = E_{m, (0, 1)^{l}(1,0)^k (1, 1)^p}$. Let $\mathscr{B}_1 = \{ (x_1, d_1), \ldots, (x_{l_1}, d_{l_1}) \}$ and  $\mathscr{B}_2 = \{(x_1', d_1'), \ldots, (x_{l_2}', d_{l_2}')\}$ be such that ${\mathscr{B}_1 \neq \mathscr{B}_2}$. Define $m_i, m_i'$ as in Lemma~\ref{lemmam}.
  
 Without loss of generality assume that  
 \[
 E_{m, (0, 1)^{l}(1,0)^k (1, 1)^p}(\mathscr{B}_1) \geq  E_{m, (0, 1)^{l}(1,0)^k (1, 1)^p}(\mathscr{B}_2)
 \] 
 and  
 \[
 E_{m, (0, 1)^{l}(1,0)^k (1, 1)^p}(\mathscr{B}_1) = \sum_{i=1}^p(m_i+d_i) + \sum_{i=p+1}^{p+k} m_i + \sum_{i=p+k+1}^{p+k+l} d_i.
 \]
 If $l_1, l_2 < p+k+l$, we add 0 length intervals to both barcodes.

Let $\theta$ be a bijection where the penalty is minimal, i.e.\ where ${P_\infty(\theta) = \textrm{d}_\infty(\mathscr{B}_1, \mathscr{B}_2)}$. Assume that $\theta$ matches $(x_1, d_1)$ with $(x_1', d_1')$, $(x_2, d_2)$ with $(x_2', d_2')$, \ldots, $(x_{p+k+l}, d_{p+k+l})$ with $(x_{p+k+l}', d_{p+k+l}')$. 
Recall that for all $i$ in this matching,
\[
|\frac{d_i-d_i'}{2}| \leq \max_{i=1, \ldots, m}(|x_i-x_i'|, |d_i-d_i'+x_i-x_i'|).
\]
Let's also check what happens if $d_i'=0$.  In this case, $(x_i, d_i)$ is matched to a 0 length barcode and
\[
d_i \leq  2\textrm{d}_\infty(\mathscr{B}_1, \mathscr{B}_2),\quad m_i \leq md_i \leq 2m\textrm{d}_\infty(\mathscr{B}_1, \mathscr{B}_2)
\]
and
\[
d_i + m_i \leq (2+2m)\textrm{d}_\infty(\mathscr{B}_1, \mathscr{B}_2).
\]
Let $M=\max\{1, m\}$. Using Lemma~\ref{lemmam} and the above inequalities 
\[
\begin{array}{lcl}
 E(\mathscr{B}_1)-E(\mathscr{B}_2)  & = & \sum_{i=1}^p(m_i+d_i) + \sum_{i=p+1}^{p+k} m_i + \sum_{i=p+k+1}^{p+k+l} d_i-E(\mathscr{B}_2) \\
 & \leq &\sum_{i=1}^p(m_i-m_i'+d_i-d_i') + \sum_{i=p+1}^{p+k} (m_i-m_i') + \sum_{i=p+k+1}^{p+k+l} (d_i-d_i')\\
  & = &2| \sum_{i=1}^p\frac{m_i-m_i'}{2}+ \sum_{i=1}^p\frac{d_i-d_i'}{2} + \sum_{i=p+1}^{p+k} \frac{m_i-m_i'}{2} + \sum_{i=p+k+1}^{p+k+l} \frac{d_i-d_i'}{2}|\\
    & \leq  &2( \sum_{i=1}^p|\frac{m_i-m_i'}{2}|+ \sum_{i=1}^p|\frac{d_i-d_i'}{2}| + \sum_{i=p+1}^{p+k} |\frac{m_i-m_i'}{2}| + \sum_{i=p+k+1}^{p+k+l} |\frac{d_i-d_i'}{2}|)\\
    &\leq & 2 (2pMP_\infty(\theta) + 2pP_\infty(\theta) +2kMP_\infty(\theta) + 2lP_\infty(\theta))\\
    &\leq & 2 (2pM+2p+2kM+2l) \textrm{d}_\infty(\mathscr{B}_1, \mathscr{B}_2).\\
\end{array}
\] 
This proves that $E$ is Lipschitz. In Proof of Theorem~\ref{maxplus} we showed that stable functions on the barcode space are preserved under taking sums, maxima and minima. Since $E_{m, (e_{1, 1}, e_{1, 2}), \ldots, (e_{n, 1}, e_{n, 2})}$ are stable as any $F\in \mathscr{G}$ is composed of taking sums, maxima and minima of $E_{m, (e_{1, 1}, e_{1, 2}), \ldots, (e_{n, 1}, e_{n, 2})}$.
\end{proof}

\begin{theorem}[Wasserstein stability of functions in $\mathscr{G}$]\label{tropratstab1}
If $F \in \mathscr{G}$, then a constant $C$ exists such that
\[
|F(\mathscr{B}_1) - F(\mathscr{B}_2)| \leq C \textrm{d}_q(\mathscr{B}_1, \mathscr{B}_2)
\]
for any pair of barcodes $\mathscr{B}_1$ and  $\mathscr{B}_2$.
\end{theorem}
\begin{proof}
 We denote the function $E_{m, (0, 1)^{l}(1,0)^k (1, 1)^p}$ by $E$. Let $\mathscr{B}_1 = \{ (x_1, d_1), \ldots, (x_{l_1}, d_{l_1}) \}$ and ${\mathscr{B}_2 = \{(x_1', d_1'), \ldots, (x_{l_2}', d_{l_2}')\}}$ be such that ${\mathscr{B}_1 \neq \mathscr{B}_2}$. Without loss of generality assume that  
 \[
 E_{m, (0, 1)^{l}(1,0)^k (1, 1)^p}(\mathscr{B}_1) \geq  E_{m, (0, 1)^{l}(1,0)^k (1, 1)^p}(\mathscr{B}_2)
 \] 
 and  
 \[
 E_{m, (0, 1)^{l}(1,0)^k (1, 1)^p}(\mathscr{B}_1) = \sum_{i=1}^p(m_i+d_i) + \sum_{i=p+1}^{p+k} m_i + \sum_{i=p+k+1}^{p+k+l} d_i.
 \]
 If $l_1, l_2 < p+k+l$, we add 0 length intervals to both barcodes.

Let $\theta$ be a bijection where the penalty is minimal, i.e.\ where ${P_\infty(\theta) = \textrm{d}_q(\mathscr{B}_1, \mathscr{B}_2)}$. Assume that $\theta$ matches $(x_1, d_1)$ with $(x_1', d_1')$, $(x_2, d_2)$ with $(x_2', d_2')$, \ldots, $(x_{p+k+l}, d_{p+k+l})$ with $(x_{p+k+l}', d_{p+k+l}')$. 
Recall that for all $i$ in this matching,
\[
|\frac{d_i-d_i'}{2}|^q \leq \max_{i=1, \ldots, m}(|x_i-x_i'|, |d_i-d_i'+x_i-x_i'|)^q
\]
since $x \mapsto x^q$ is increasing for $x >0$.
As before, if $d_i'=0$, $(x_i, d_i)$ is matched to a 0 length barcode and
\[
d_i \leq  2\textrm{d}_q(\mathscr{B}_1, \mathscr{B}_2),\quad m_i \leq md_i \leq 2m\textrm{d}_q(\mathscr{B}_1, \mathscr{B}_2)
\]
and
\[
d_i + m_i \leq (2+2m)\textrm{d}_q(\mathscr{B}_1, \mathscr{B}_2).
\]
Let $M=\max\{1, m\}$. Using Lemma~\ref{lemmam} and the above inequalities, we get:
\[
\begin{array}{lcl}
|E(\mathscr{B}_1)-E(\mathscr{B}_2)|^q & = & (\sum_{i=1}^p(m_i+d_i) + \sum_{i=p+1}^{p+k} m_i + \sum_{i=p+k+1}^{p+k+l} d_i-E(\mathscr{B}_2))^q \\
 & \leq &(\sum_{i=1}^p(m_i-m_i'+d_i-d_i') + \sum_{i=p+1}^{p+k} (m_i-m_i') + \sum_{i=p+k+1}^{p+k+l} (d_i-d_i'))^q\\
  & = &2^q| \sum_{i=1}^p\frac{m_i-m_i'}{2}+ \sum_{i=1}^p\frac{d_i-d_i'}{2} + \sum_{i=p+1}^{p+k} \frac{m_i-m_i'}{2} + \sum_{i=p+k+1}^{p+k+l} \frac{d_i-d_i'}{2}|^q\\
    & \leq  &2^q( \sum_{i=1}^p|\frac{m_i-m_i'}{2}|+ \sum_{i=1}^p|\frac{d_i-d_i'}{2}| + \sum_{i=p+1}^{p+k} |\frac{m_i-m_i'}{2}| + \sum_{i=p+k+1}^{p+k+l} |\frac{d_i-d_i'}{2}|)^q\\
        & \leq  &2^q (2p+k+l)^{q-1}(2pM+2p+2kM+2l)^q P_q(\theta)^q\\
        &\leq & 2^q (2pM+2p+2kM+2l)^{2q} \textrm{d}_q(\mathscr{B}_1, \mathscr{B}_2)^q.\\
\end{array}
\] 
The first inequality holds since $\sum_{i=1}^p(m_i'+d_i') + \sum_{i=p+1}^{p+k} m_i' + \sum_{i=p+k+1}^{p+k+l} d_i' \leq E (\mathscr{B}_2)$. The last inequality uses H{\H o}lder's inequality. Taking the $q$-th root finishes the proof.

We conclude that $E$ is Lipschitz. In Proof of Theorem~\ref{maxplus} we showed that stable functions on the barcode space are preserved under taking sums, maxima and minima. Since $E_{m, (e_{1, 1}, e_{1, 2}), \ldots, (e_{n, 1}, e_{n, 2})}$ are stable, $F\in \mathscr{G}$ is also stable as it is composed of taking sums, maxima and minima of $E_{m, (e_{1, 1}, e_{1, 2}), \ldots, (e_{n, 1}, e_{n, 2})}$.
\end{proof}

\section{Classifying Digits with Tropical Coordinates}
Adcock et al.~\cite{algfn} used polynomial coordinates to classify digits from the MNIST database~\cite{MNISTdata} of handwritten digits. In this section we compare classification results they obtained with mine, which were classified using tropical coordinates. Aaron Adcock provided the matlab code needed to convert digital images into filtrations.

While homology itself cannot distinguish between the digits - 1, 5, and 7 never have loops, 0, 6, 9 always have loops, 8 has two loops, while 2, 3, 4 might or might not have loops, depending on style - we can use persistent homology as a measurement of shape. Figure~\ref{digits} shows the first 100 digits of the database.
\begin{figure}[h!]
  \begin{center}
    \includegraphics[scale=0.4]{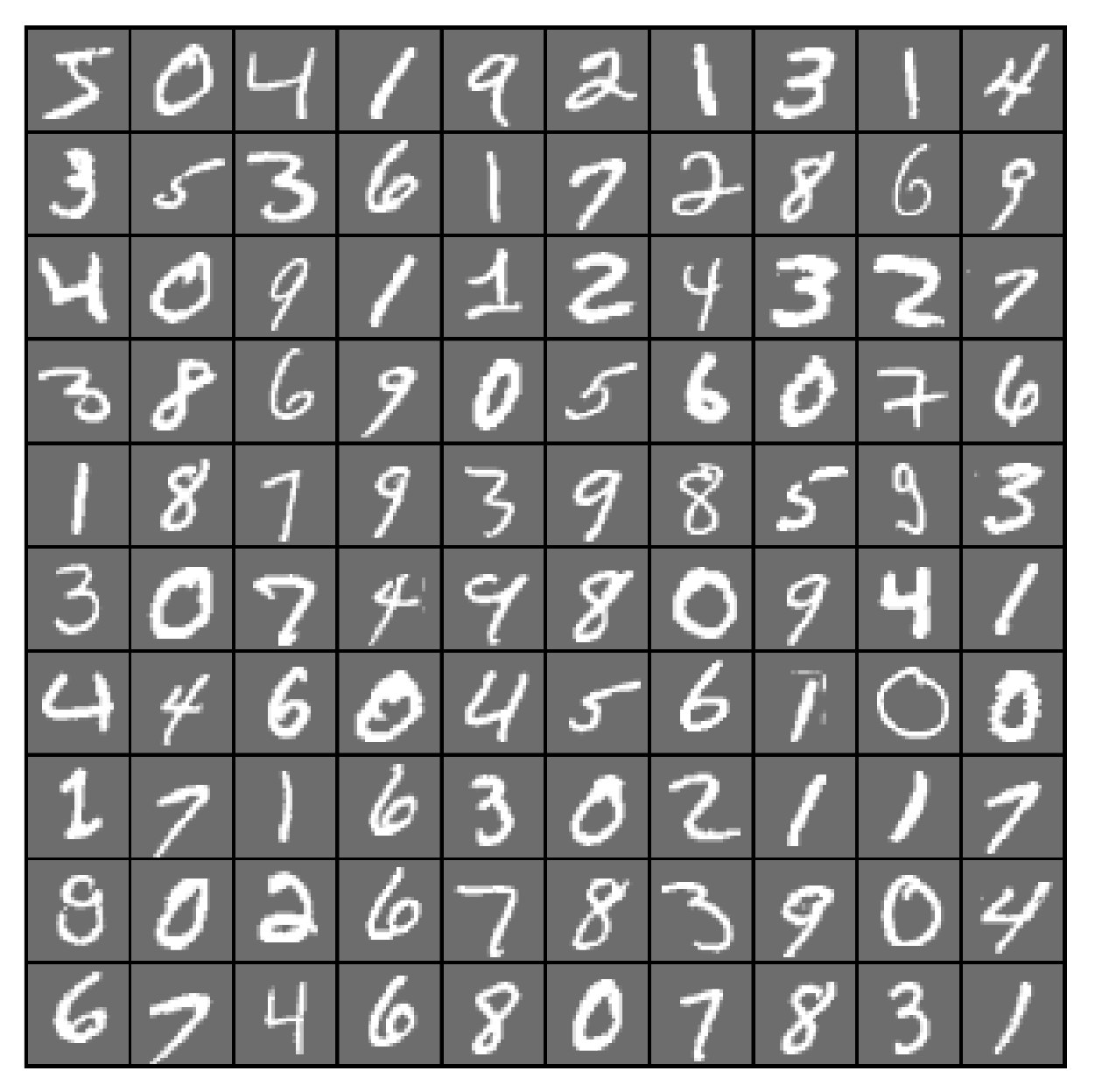}
      \caption{The first 100 images of the MNIST database.}\label{digits}
  \end{center}
  \end{figure}
The original black and white images were first normalized, scaled into a $20\times 20$ pixel bounding box and anti-aliased, which introduced grayscale levels. Pixel values are 0 to 255, where 0 means background (white), 255 means foreground (black).

Following Collins et al.~\cite{Collins:2004:BSD:2386332.2386363}, we first threshold (setting pixel values greater than 100 to 1 and the rest to 0) to produce a binary image. We construct four filtrations as follows. For each pixel we add a vertex, for any pair of adjacent pixels (diagonals included) an edge and for any triple of adjacent pixels a 2-simplex. We sweep across the rows from the left and the right and across the columns from top to bottom and vice versa. This adds spatial information into what would otherwise be a purely topological measurement. We take both Betti 0 and Betti 1.

This extra spatial information reveals the location of various topological features. For example, though a `9' and `6' both have one connected component and one loop, the loop will appear at different locations in the 1-dimensional homology top-down sweep for the `9' and `6' (see Figure~\ref{17odd}). In digits with no loops 0-dimensional homology right to left sweep distinguishes `3' from other digits (see Figure~\ref{02even}).  
  \begin{figure}[h!]
  \begin{center}
    \begin{tabular}{cc}
     \includegraphics[scale=0.1]{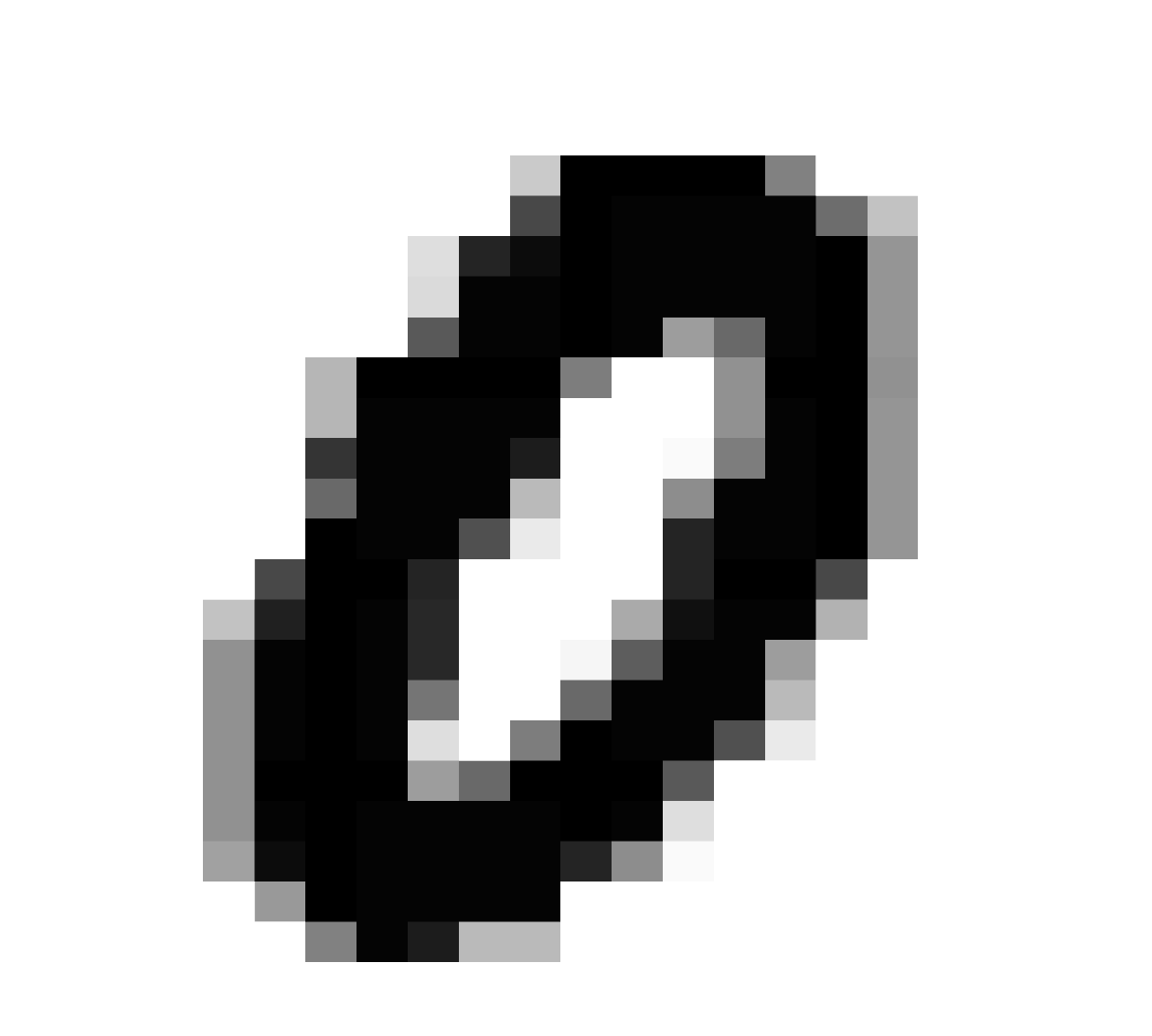}       &       \includegraphics[scale=0.3]{0upsweep1.pdf}\\
           &\\
      \includegraphics[scale=0.1]{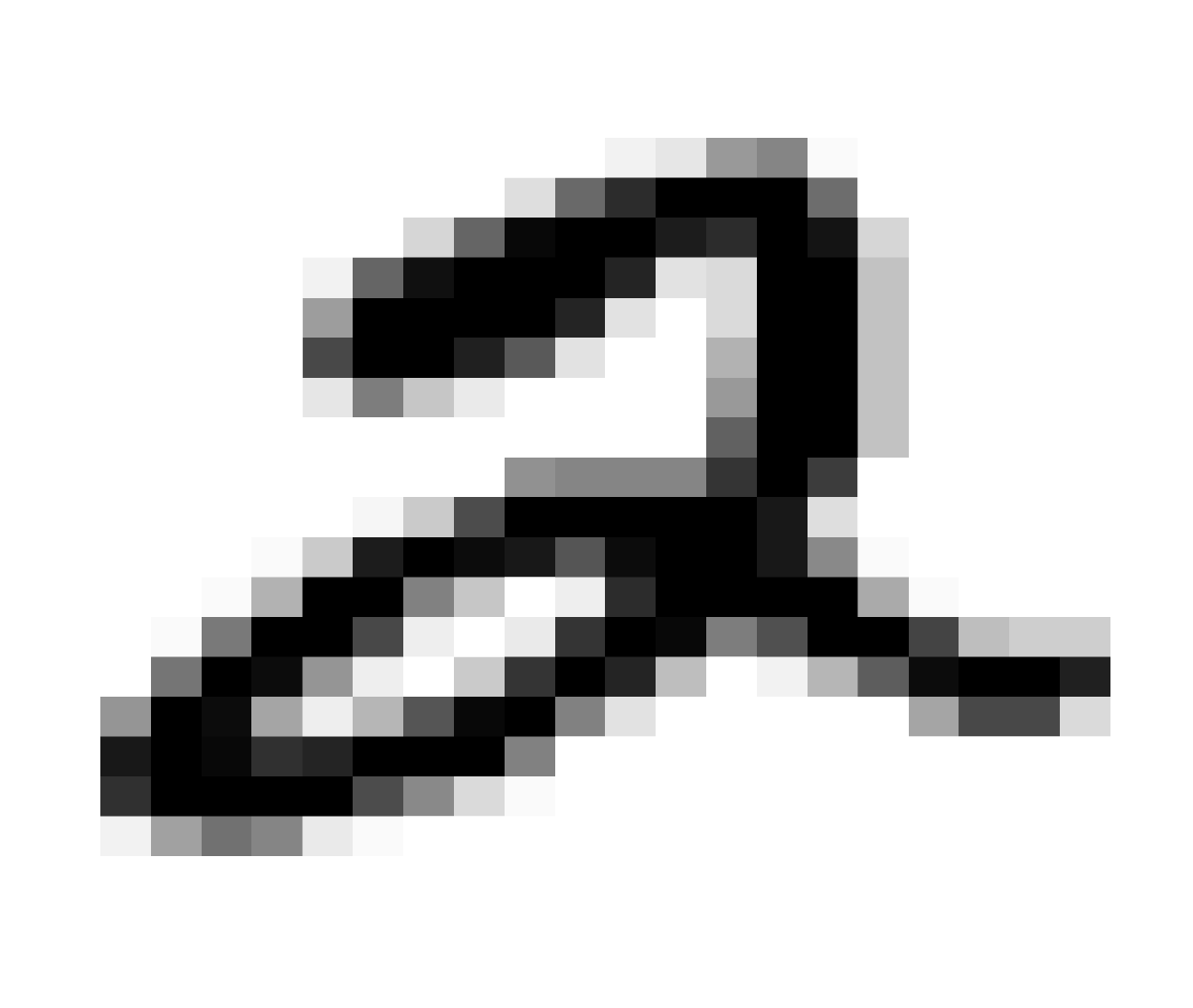}    &  \includegraphics[scale=0.3]{2upsweep1.pdf}\\
            &\\
        \includegraphics[scale=0.1]{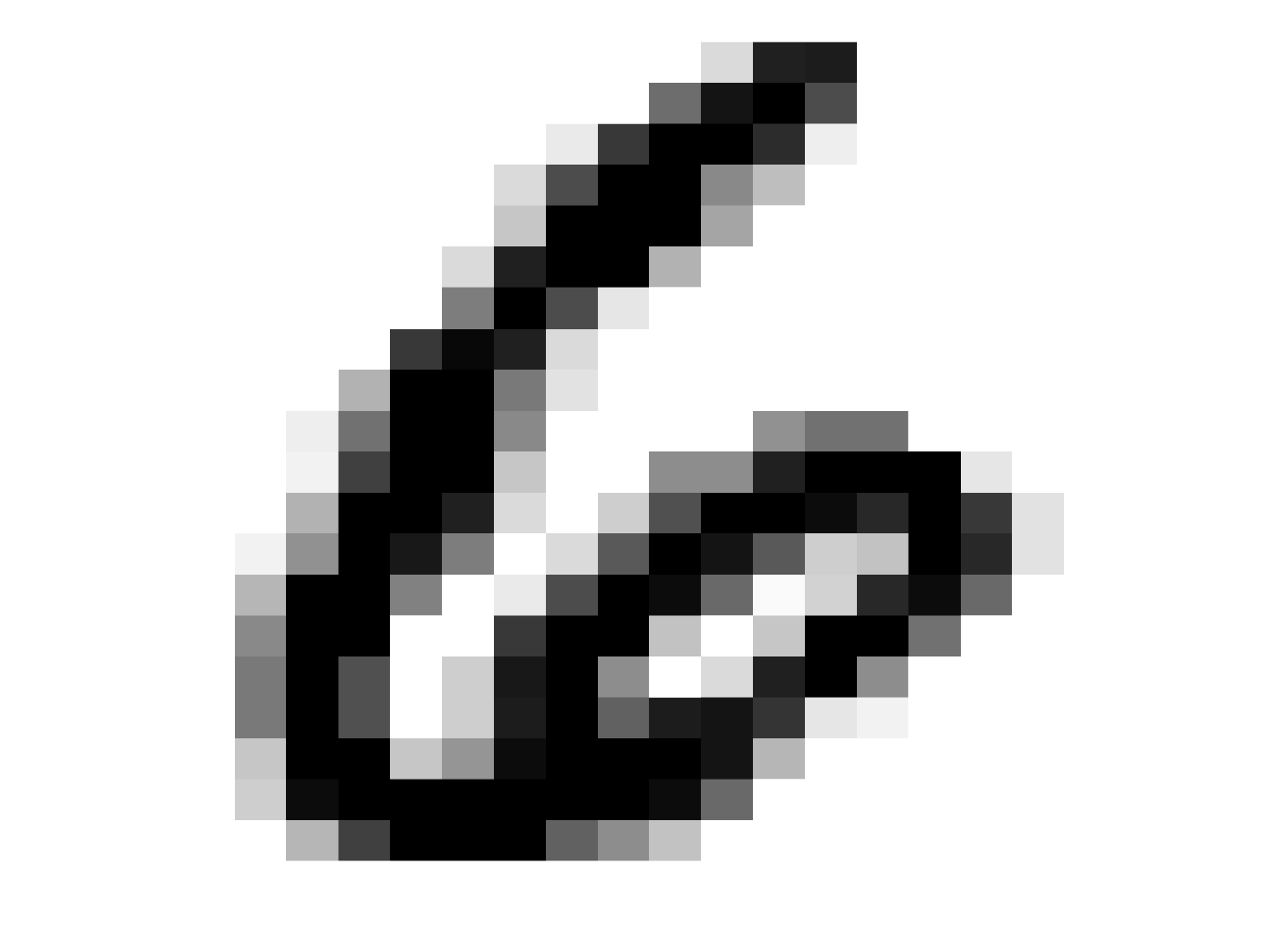}    &          \includegraphics[scale=0.3]{6upsweep1.pdf}\\
              &\\
          \includegraphics[scale=0.1]{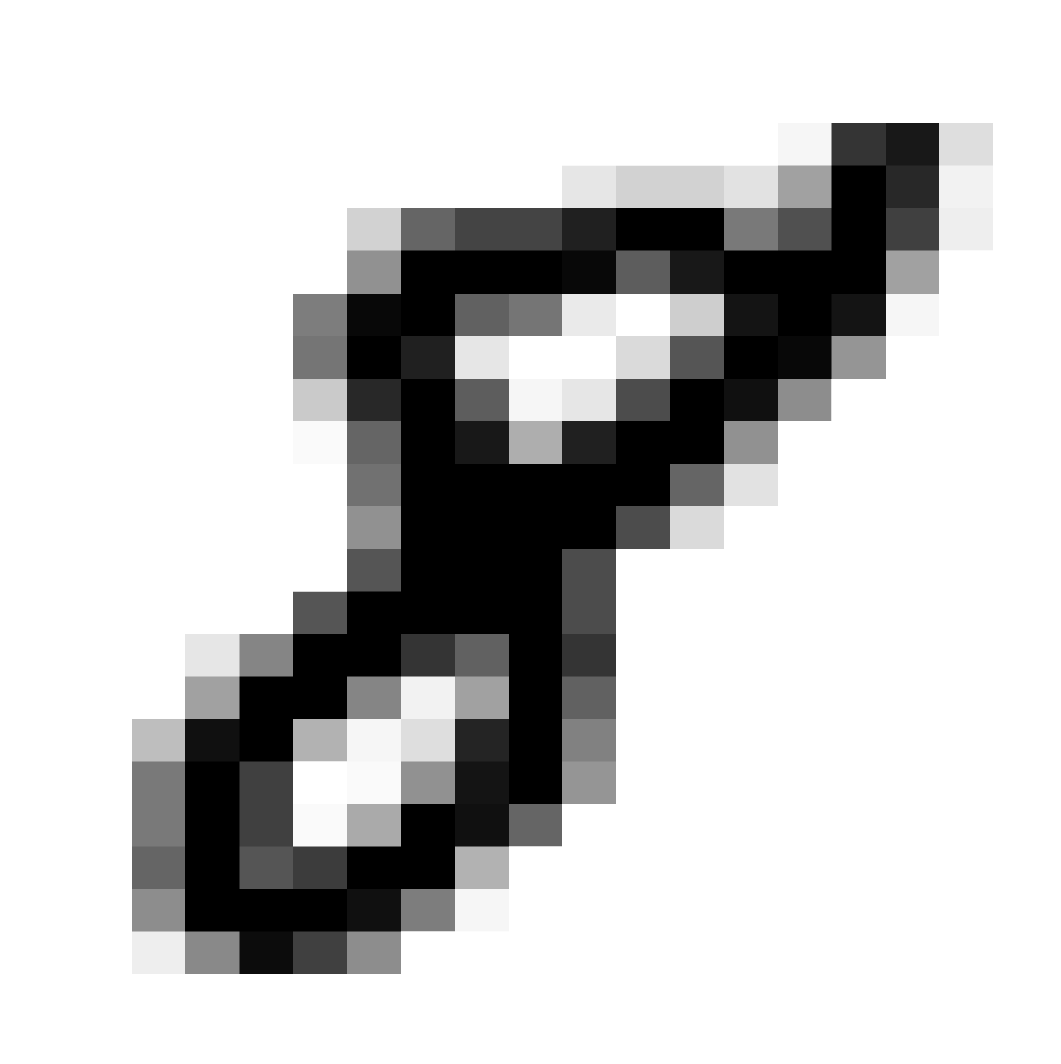}    &  \includegraphics[scale=0.3]{8upsweep1.pdf}\\
                &\\
          \includegraphics[scale=0.1]{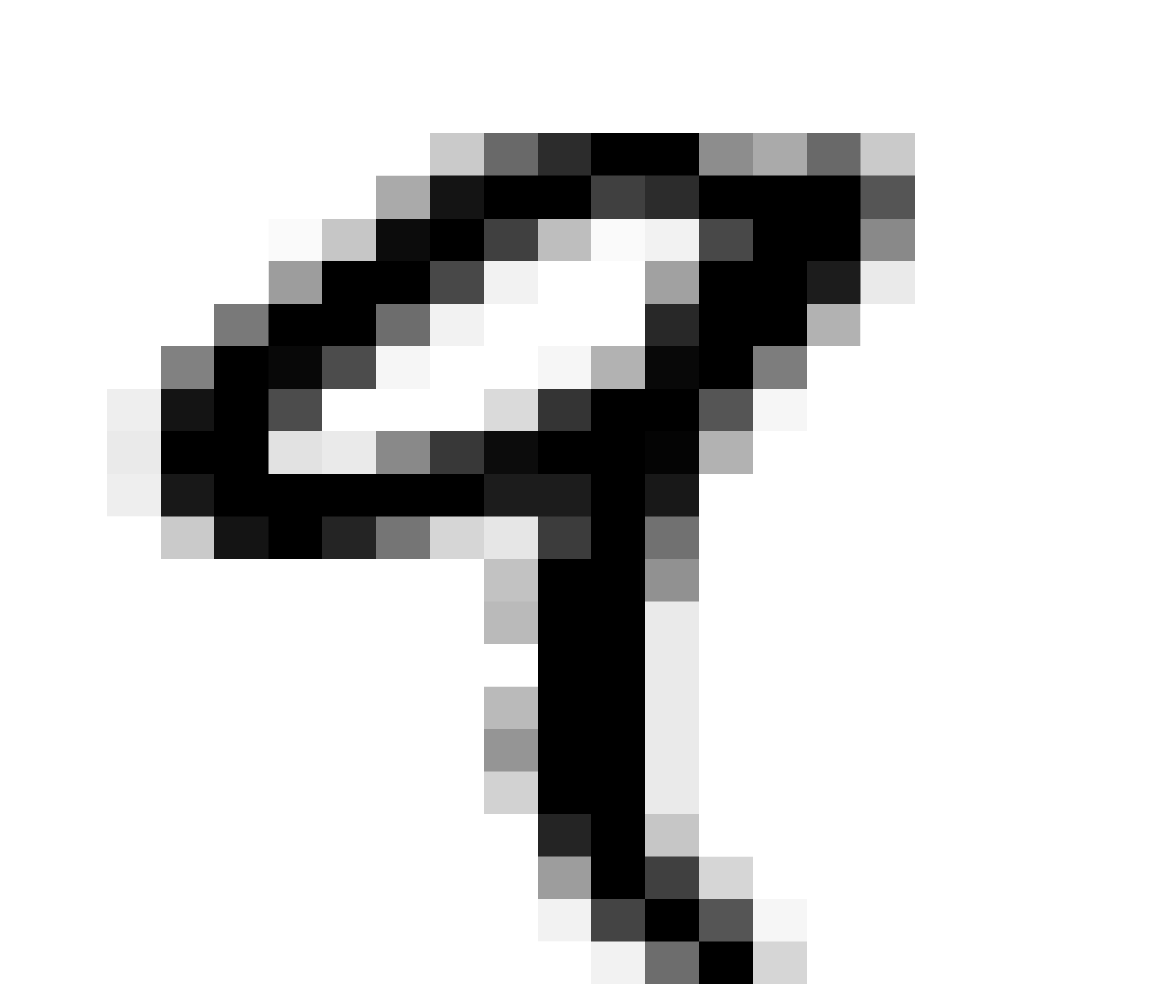}      &       \includegraphics[scale=0.3]{9upsweep1.pdf}       \\  
 \end{tabular}
      \caption{1-dimensional homology bottom to top sweep for `0', `2', `6', `8' and `9'.}\label{02even}
  \end{center}
  \end{figure}

\begin{figure}[h!]
  \begin{center}
  \begin{tabular}{cc}
      \includegraphics[scale=0.1]{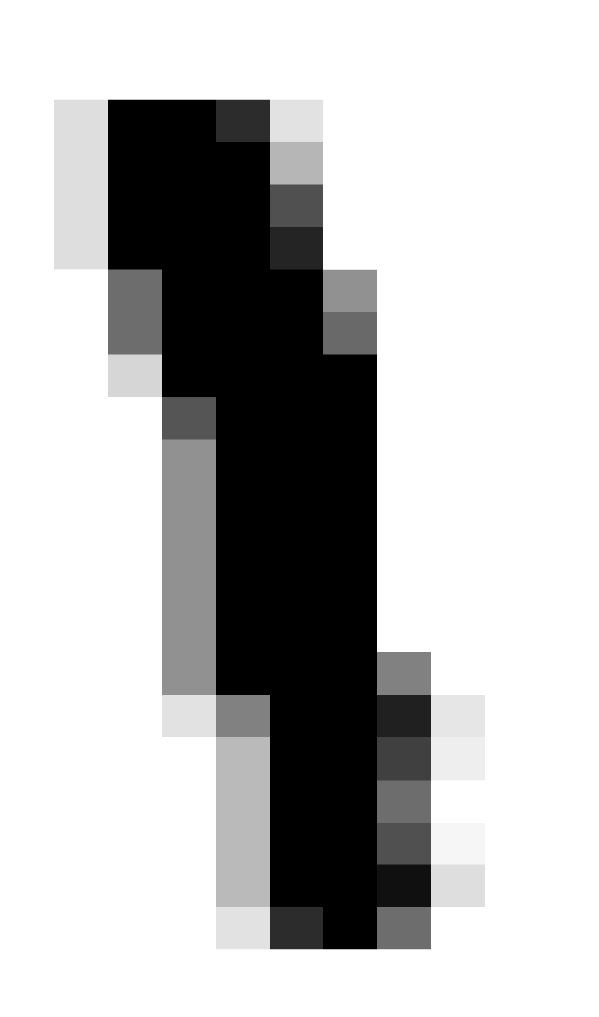}     &  \includegraphics[scale=0.3]{1rightsweep0.pdf}\\
      &\\
        \includegraphics[scale=0.1]{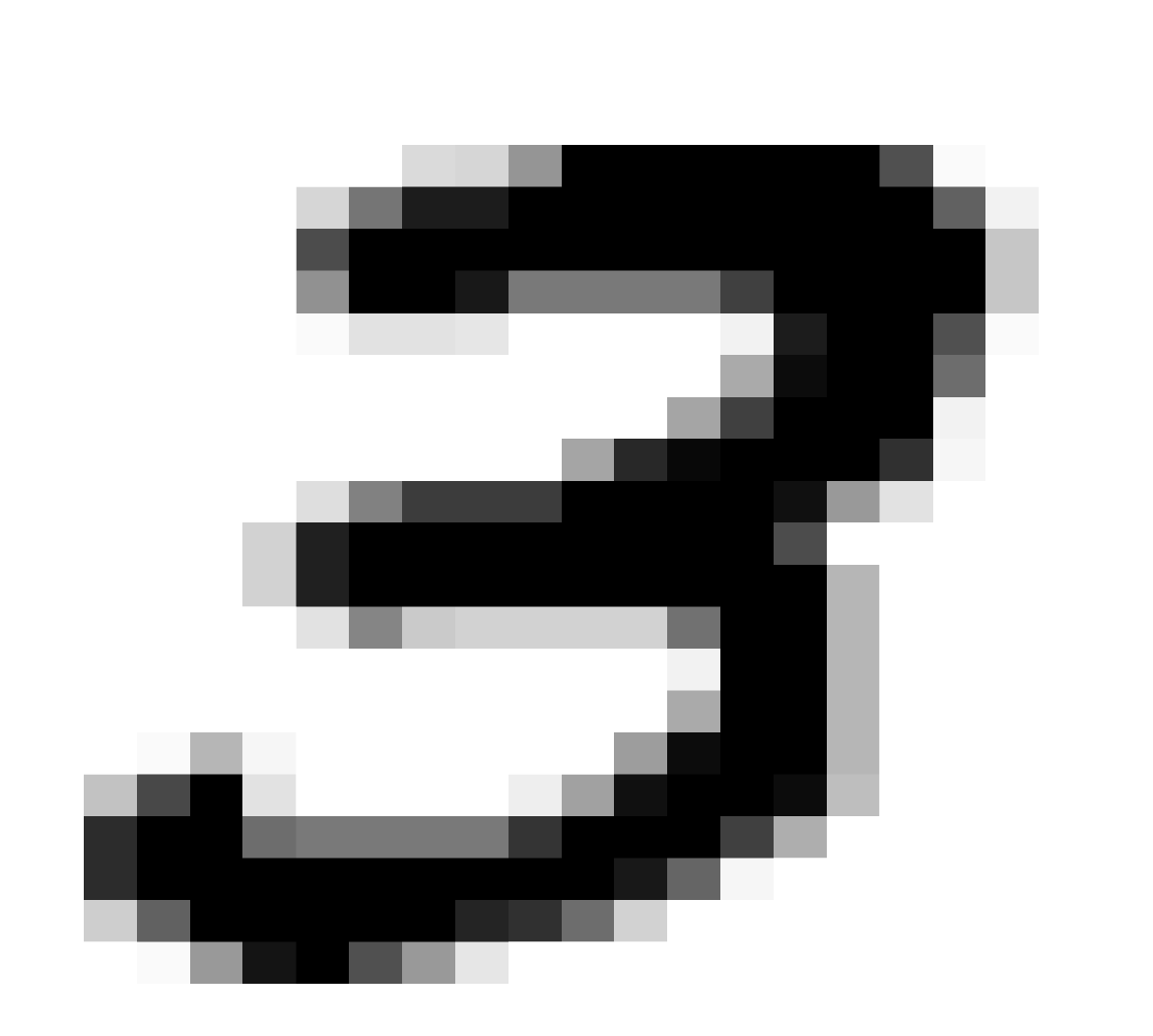}     &    \includegraphics[scale=0.3]{3rightsweep0.pdf}\\   
        &\\
               \includegraphics[scale=0.1]{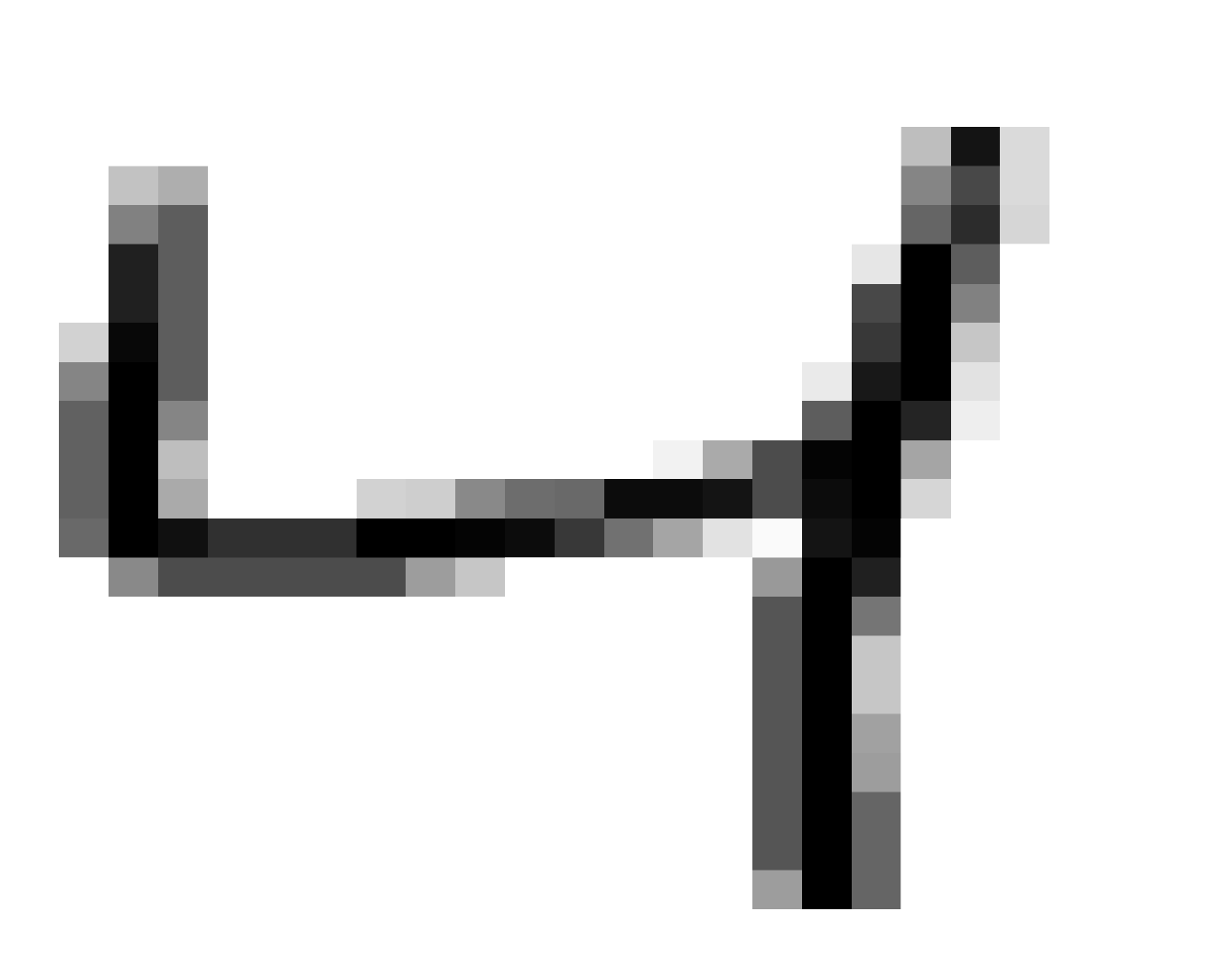}   &      \includegraphics[scale=0.3]{4rightsweep0.pdf}\\
               &\\
      \includegraphics[scale=0.1]{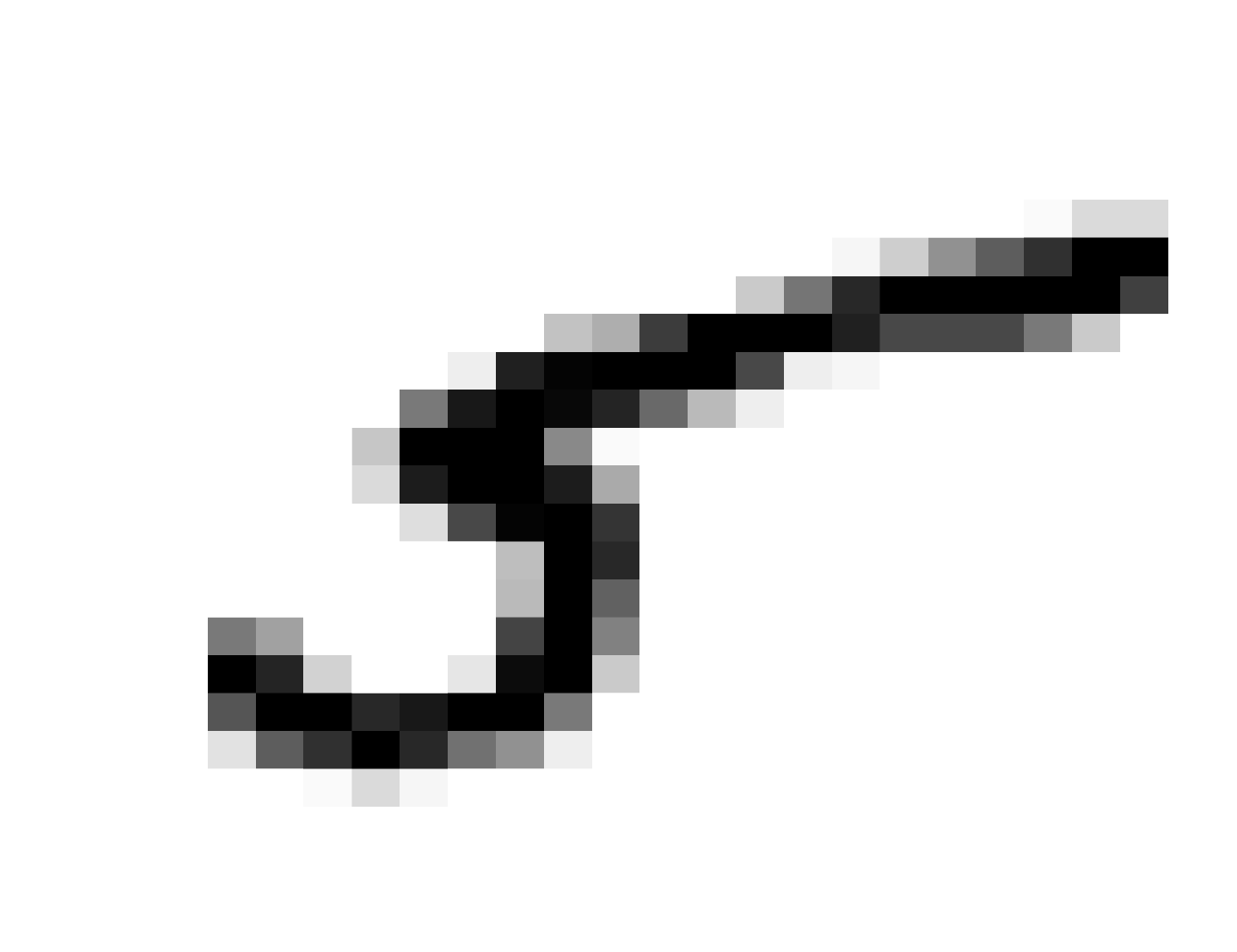}   &    \includegraphics[scale=0.3]{5rightsweep0.pdf}\\
      &\\
          \includegraphics[scale=0.1]{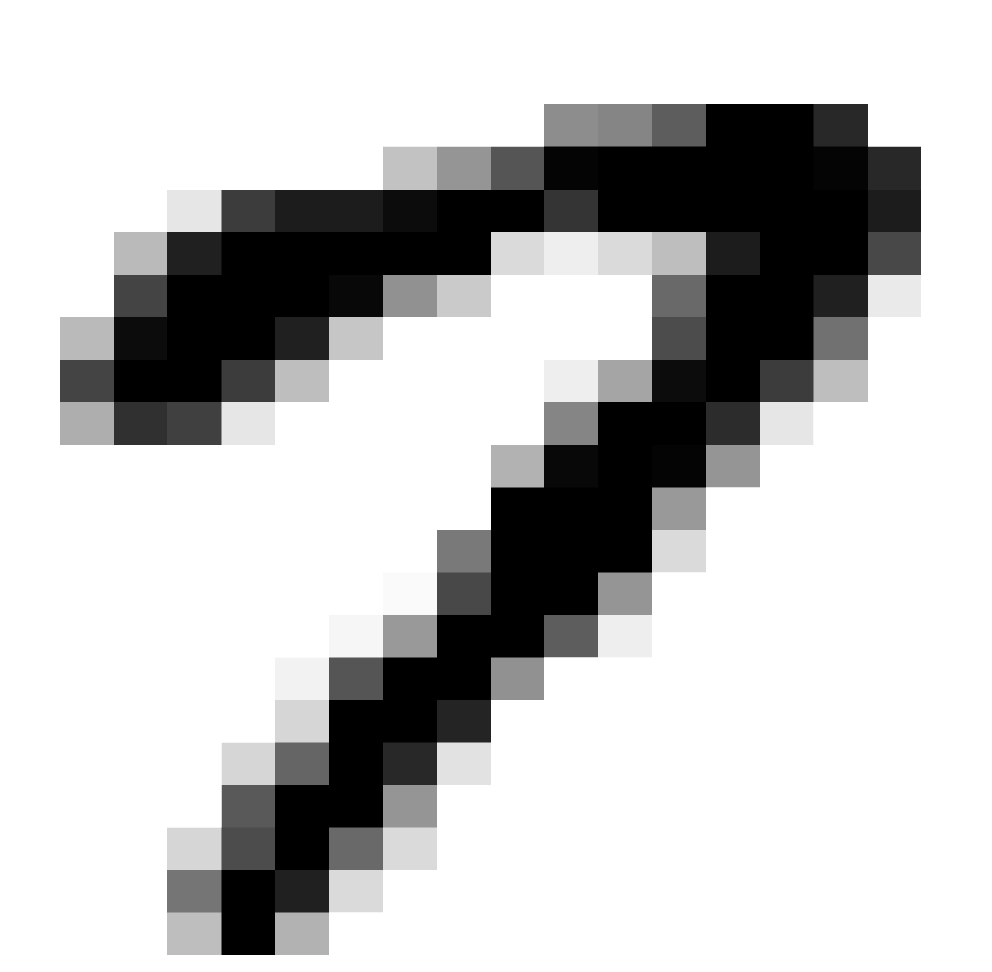}   &    \includegraphics[scale=0.34]{7rightsweep0.pdf}
      \end{tabular}
      \caption{0-dimensional homology right to left sweep for `1', `3', `4', `5' and `7'.}\label{17odd}
  \end{center}
  \end{figure}
We can use different methods for turning barcodes into vectors. Adcock et al. selected four features,
\[
\begin{array}{l}
\sum_i x_i(y_i-x_i) \\
 \sum_i(y_{\max} - y_i)(y_i- x_i)\\
 \sum_ix_i^2 (y_i-x_i)^4\\
 \sum_i(y_{\max} - y_i)^2(y_i-x_i)^4
\end{array}
\]
which when applied to the four sweeps, each with a 0-dimensional and 1-dimensional barcode, gives a feature vector of total size 32. We used command \texttt{fitcecoc} in matlab to get an error-correcting output codes (ECOC) multiclass model~\cite{MATLAB:2017}. This model was trained using support vector machine (SVM)~\cite{Cortes1995}. We obtained the best results using the Gaussian kernel. As is typical when using a SVM, we scaled each coordinate such that the values were between 0 and 1. To measure the classification accuracy we used 100-fold cross-validation. See Table~\ref{ordclass} for results.

\begin{table}[h!]
\begin{center}
\begin{tabular}{| c | c | c |}
\hline			
1000 digits & 5000 digits & 10000 digits \\ \hline
87.5\% &  90.04\% & 91.04\% \\ \hline  
\end{tabular}
\caption{Classification accuracy using ordinary polynomial coordinates.}\label{ordclass}
\end{center}
 \end{table}         
Using the following max-plus type coordinates 
\[
\begin{array}{ll}
\max_{i} d_i & \max_{i<j} (d_i + d_j) \\
\max_{i<j<k} (d_i + d_j+d_k) & \max_{i<j<k<l} (d_i + d_j+d_k+d_l) \\
 \sum_i d_i  & \sum_i \min(28d_i, x_i)  \\
 \sum_i (\max_i(\min(28d_i, x_i) +d_i) -(\min(28d_i, x_i) +d_i) ) &   \\
\end{array}
\]
yields slightly better results (see Table~\ref{maxclass}).
\begin{table}[h!]
\begin{center}
\begin{tabular}{| c | c | c |}
\hline			
1000 digits & 5000 digits & 10000 digits \\ \hline
 87.70\% & 91.36\% & 92.41\% \\ \hline  
\end{tabular}
\caption{Classification accuracy using max-plus type coordinates.}\label{maxclass}
\end{center}
 \end{table}  
Note that we used many functions involving sums of lengths of intervals. These yielded the best results, which is perhaps not surprising since when using persistent homology and interpreting the barcode, we assign importance to features depending on over what range of parameters they persist.
 
This method just demonstrates how one can use persistent homology with other machine learning algorithms and does not outperform existing classification algorithms. Figure~\ref{misclass} shows examples of digits that were not correctly classified. 
\begin{figure}[h!]
  \begin{center}
    \includegraphics[scale=1.3]{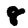} \includegraphics[scale=1.3]{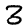} \includegraphics[scale=1.3]{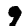}
     \includegraphics[scale=1.3]{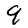}   \includegraphics[scale=1.3]{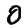}   \includegraphics[scale=1.3]{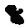}
     
      \caption{Common Misclassifications.}\label{misclass}
  \end{center}
  \end{figure}
The most common confusion is between a `5' and a `2' written with no loop. Other common confusions occur when topological changes occurred to the digit, for example when `8' is written with no loops, etc.

These examples also show the power of combining topology with geometry, and in particular demonstrate how coordinates can serve as a method for organizing the collection of all barcodes, and therefore any database whose members produce barcodes. 
 They are also stable with respect to the bottleneck and Wasserstein distances.

{\large Acknowledgement}

The author thanks Gunnar Carlsson, Gregory Brumfiel, Davorin Le\v{s}nik and the referees for helpful discussions.

\printbibliography
 \end{document}